\documentclass{amsart}
\usepackage[ps2pdf]{hyperref}
\usepackage{amsmath,amssymb}
\usepackage{bbold}
\usepackage{psfrag}
\usepackage{graphicx}
\usepackage{subfigure}
\usepackage{leftidx}
\usepackage[all]{xy}
\usepackage[latin1]{inputenc}
\usepackage{enumerate}

\newtheorem{thm}{Theorem}[section]
\newtheorem{cor}[thm]{Corollary}
\newtheorem{lem}[thm]{Lemma}

\theoremstyle{definition}

\newtheorem{rem}[thm]{Remark}
\newtheorem{exa}[thm]{Example}

\newcommand{\co}{\colon}
\newcommand{\id}{\mathrm{id}}
\newcommand{\kk}{\Bbbk}
\newcommand{\kt}{$\Bbbk$\nobreakdash-\hspace{0pt}}

\newcommand{\slo}{\mathrm{sl}}
\newcommand{\SLO}{\mathrm{SL}}
\newcommand{\labela}{\renewcommand{\labelenumi}{{\rm (\alph{enumi})}}}
\newcommand{\labeli}{\renewcommand{\labelenumi}{{\rm (\roman{enumi})}}}

\newcommand{\opp}{\mathrm{op}}

\newcommand{\cc}{\mathcal{C}}
\newcommand{\bb}{\mathcal{B}}

\newcommand{\dd}{\mathcal{D}}
\newcommand{\aaa}{\mathcal{A}}

\newcommand{\ff}{\mathcal{F}}

\newcommand{\uu}{\mathcal{U}}

\newcommand{\zz}{\mathcal{Z}}

\newcommand{\RR}{\mathbb{R}}

\newcommand{\Rt}{$\mathrm{R}$\nobreakdash-\hspace{0pt}}
\newcommand{\ti}{\mbox{-}\,}
\newcommand{\un}{\mathbb{1}}

\newcommand{\End}{\mathrm{End}}
\newcommand{\Hom}{\mathrm{Hom}}

\newcommand{\tr}{\mathrm{tr}}

\newcommand{\rank}{\mathrm{rank}}
\newcommand{\lev}{\mathrm{ev}}
\newcommand{\rev}{\widetilde{\mathrm{ev}}}
\newcommand{\lcoev}{\mathrm{coev}}
\newcommand{\rcoev}{\widetilde{\mathrm{coev}}}

\newcommand{\ldual}[1]{{#1}^*}

\newcommand{\adjunct}[2]{\!\!\raisebox{.6ex}{\xymatrix{\ar@/^.4pc/[r]^{#1}  &  \ar@/^.4pc/[l]^{#2}}}\!\!}
\newcommand{\rsdraw}[3]{\raisebox{-#1\height}{\scalebox{#2}{\includegraphics{#3.eps}}}}
\providecommand{\bysame}{\leavevmode\hbox to3em{\hrulefill}\thinspace}

\usepackage[dvipsnames]{xcolor}

\begin{document}
\title{On the center of fusion categories}
\author[A. Brugui\`eres]{Alain Brugui\`eres}
\author[A. Virelizier]{Alexis Virelizier}
\email{bruguier@math.univ-montp2.fr \and virelizi@math.univ-montp2.fr}

\date{\today}

\begin{abstract} M\"{u}ger proved in 2003  that the center of a spherical fusion category~$\cc$ of non-zero dimension over an algebraically closed field is a modular fusion category whose dimension is the square of that of $\cc$. We generalize this theorem to a pivotal fusion category $\cc$ over an
arbitrary commutative ring $\kk$, without any condition on the dimension of the category. (In this generalized setting, modularity is understood as 2-modularity in the sense of Lyubashenko.) Our proof is based on an explicit description of the Hopf algebra structure of  the coend of the center of $\cc$.
Moreover  we show that the dimension of $\cc$ is invertible in $\kk$ if and only if any object of the center of $\cc$ is a retract of a `free' half-braiding. As a consequence,   if $\kk$ is a field, then the center of $\cc$ is  semisimple (as an abelian category) if and only if the dimension of $\cc$ is non-zero.
If in addition $\kk$ is algebraically closed, then this condition implies that the center is a fusion category, so that we recover M\"{u}ger's result.
\end{abstract}
\maketitle

\setcounter{tocdepth}{1} \tableofcontents

\section*{Introduction}
Given a monoidal category~$\cc$, Joyal and Street~\cite{JS}, Drinfeld (unpublished),  and Majid~\cite{Maj}  defined a braided category $\zz(\cc)$, called the center of $\cc$, whose objects are half-braidings of $\cc$.
M\"{u}ger~\cite{Mu} showed that the center $\zz(\cc)$ of a spherical fusion category $\cc$ of non-zero dimension over an algebraically closed field $\kk$ is a modular fusion category, and that the dimension of $\zz(\cc)$ is the square of that of $\cc$. M\"{u}ger's proof of this remarkable result relies on algebraic constructions due to Ocneanu (such as the `tube' algebra) and involves the construction  of a weak monoidal Morita equivalence
between $\zz(\cc)$ and $\cc \otimes \cc^\opp$. The modularity of the center is of special interest in $3$-dimensional quantum topology, since spherical fusion categories and modular categories are respectively the algebraic input for the construction of  the Turaev-Viro/Barrett-Westbury invariant  and of the Reshetikhin-Turaev invariant.  Indeed it has been shown recently in \cite{TVi} (see also \cite{Ba}) that, under the hypotheses of M\"{u}ger's theorem,
the  Barrett-Westbury generalization of the Turaev-Viro invariant for $\cc$ is equal to the Reshetikhin-Turaev invariant for $\zz(\cc)$.

In this paper, we generalize M\"{u}ger's theorem to pivotal fusion categories over an arbitrary   commutative ring.   More precisely, given a pivotal fusion category $\cc$ over a commutative ring $\kk$, we prove the following:
\begin{enumerate}
\labeli
\item The center $\zz(\cc)$ of $\cc$ is always modular (but not necessarily semisimple) and has dimension $\dim(\cc)^2$.
\item The scalar $\dim(\cc)$ is invertible in $\kk$ if and only if every half braiding is a retract of a so-called \emph{free half braiding}.
\item If $\kk$ is a field, then $\zz(\cc)$ is abelian semisimple if and only if $\dim(\cc)\neq 0$.
\item If $\kk$ is an algebraically closed field, then $\zz(\cc)$ is fusion if and only if $\dim(\cc)\neq 0$.
\end{enumerate}
Our proof is different from that of M\"{u}ger.  It relies on the principle that if a braided category $\bb$ has a coend, then all the
relevant information about $\bb$ is encoded in its coend, which is a universal Hopf algebra sitting in $\bb$ and endowed with a canonical Hopf algebra pairing. For instance, modularity means that the canonical pairing is non-degenerate, and the dimension of $\bb$ is that of its coend.
In particular we do not need to introduce an auxiliary category.

The center $\zz(\cc)$ of a pivotal fusion category $\cc$ always has a coend. We provide a complete and explicit description of the Hopf algebra structure  of this coend, which enables us to exhibit an integral for the coend  and an `inverse' to the pairing. Our proofs are based on
a `handleslide'  property for pivotal fusion categories.

A general description of the coend of the center of a rigid category $\cc$, together with its structural morphisms, was given in \cite{BV3}. It is an application of the theory of Hopf monads, and in particular, of the notion of double of a Hopf monad, which generalizes the Drinfeld double of a Hopf algebra. It is based on the fact that $\zz(\cc)$ is the category of modules over a certain  quasitriangular  Hopf monad $Z$ on~$\cc$ (generalizing the braided  equivalence $\zz(\mathrm{mod}_H) \simeq \mathrm{mod}_{D(H)}$ between the center of the category of modules over a finite dimensional Hopf algebra $H$ and the category of modules over the Drinfeld double $D(H)$ of $H$).  It turns out that, when $\cc$ is a fusion category, we can make this description very explicit and in particular, we can depict the structural morphisms of the coend by means of a graphical formalism for fusion categories.

Part of the results of this paper were announced (without proofs) in \cite{BV4}, where they are used to  define and compute a   $3$-manifolds invariant of Reshetikhin-Turaev type  associated with the center of $\cc$, even when the dimension of $\cc$ is not invertible.

\subsection*{Organization of the text}
In Section~\ref{sec-pifu}, we recall definitions, notations and basic results concerning pivotal and fusion  categories over a commutative ring. A graphical formalism for representing morphisms in fusion categories is provided.
In Section~\ref{main-results}, we state the main results of this paper, that is, the description of the coend of the center of a pivotal fusion category and its structural morphisms, the modularity of the center of such a category, its dimension, and a semisimplicity criterion.
Section~\ref{sect-Big-modular} is devoted to coends, Hopf algebras in braided categories, and modular categories.
Section~\ref{sect-proofs} contains the proofs of the main results.

\section{Pivotal and fusion categories}\label{sec-pifu}

Monoidal categories are assumed to be strict. This does not lead to any loss of generality, since,
in view of MacLane's coherence theorem for monoidal categories (see \cite{ML1}), all definitions and statements remain valid for non-strict monoidal categories after insertion of the suitable canonical isomorphisms.

\subsection{Rigid categories}\label{rigid categories}
Let $\cc=(\cc,\otimes,\un)$ be a monoidal category. A \emph{left dual} of an object
  $X$ of $\cc$ is an object $\leftidx{^\vee}{\!X}{}$ of $\cc$ together
with morphisms $\lev_X \co \leftidx{^\vee}{\!X}{} \otimes X \to \un$
and $\lcoev_X \co \un \to X \otimes \leftidx{^\vee}{\! X}{}$ such that
$$
(\id_X \otimes \lev_X)(\lcoev_X \otimes \id_X)=\id_X \quad \text{and} \quad (\lev_X \otimes \id_{\leftidx{^\vee}{\! X}{}})(\id_{\leftidx{^\vee}{\! X}{}} \otimes \lcoev_X)=\id_{\leftidx{^\vee}{\! X}{}}.
$$
Similarly a \emph{right dual} of $X$ is an object $X^\vee$
with morphisms $\rev_X \co X \otimes X^\vee \to \un$ and
$\rcoev_X \co \un \to X^\vee \otimes X$ such that
$$
(\rev_X \otimes \id_X)(\id_X \otimes \rcoev_X)=\id_X \quad \text{and} \quad (\id_{X^\vee} \otimes \rev_X)(\rcoev_X \otimes \id_{X^\vee})=\id_{X^\vee}.
$$
The left and right duals of an object, if they exist, are unique up to an isomorphism (preserving the (co)evaluation morphisms.)

A monoidal category $\cc$ is \emph{rigid} (or \emph{autonomous}) if every object of $\cc$ admits a left and a right dual. The choice of left and right duals for each object of a rigid $\cc$ defines a left  dual functor $\leftidx{^\vee}{?}{}\co \cc^\opp \to \cc$ and a right  dual functor $\leftidx{}{?}{^\vee}\co \cc^\opp \to \cc$,  where $\cc^\opp$ is the opposite category to $\cc$ with opposite monoidal structure. The left and right dual functors are monoidal. Note that the actual choice of left and right duals is innocuous in the sense that different choices of left (resp.\@ right) duals define canonically monoidally isomorphic left (resp.\@ right) dual
functors.

There are canonical natural monoidal isomorphisms $\leftidx{^\vee}{(X^\vee)} \simeq X \simeq (\leftidx{^\vee}{X})^\vee$,
but in general the left and right dual functors are not monoidally isomorphic.

\subsection{Pivotal categories}\label{sec-pivot}
A rigid category $\cc$ is \emph{pivotal} (or \emph{sovereign}) if it is endowed with a monoidal isomorphism between the left and the right dual functors. We may assume that this isomorphism is the identity without loss of generality. In other words, for each
object $X$ of $\cc$, we have a \emph{dual
object}~$X^*$ and four morphisms
\begin{align*}
& \lev_X \co X^*\otimes X \to\un,  \qquad \lcoev_X\co \un  \to X \otimes X^*,\\
&   \rev_X \co X\otimes X^* \to\un, \qquad   \rcoev_X\co \un  \to X^* \otimes X,
\end{align*}
such that $(X^*,\lev_X,\lcoev_X)$ is a left
dual for $X$, $(X^*, \rev_X,\rcoev_X)$ is a right
dual for~$X$, and the induced left and right dual functors coincide as monoidal functors. In particular, the \emph{dual} $f^* \co Y^* \to X^*$ of any morphism $f\co X \to Y$ in $\cc$  is
\begin{align*}
f^*&= (\lev_Y \otimes  \id_{X^*})(\id_{Y^*}  \otimes f \otimes \id_{X^*})(\id_{Y^*}\otimes \lcoev_X)\\
 &= (\id_{X^*} \otimes \rev_Y)(\id_{X^*} \otimes f \otimes \id_{Y^*})(\rcoev_X \otimes \id_{Y^*}).
\end{align*}

In what follows, for a pivotal category $\cc$, we will suppress the duality constraints
$\un^* \cong \un$ and $X^* \otimes Y^*\cong (Y\otimes X)^* $. For
example, we will write $(f \otimes g)^*=g^* \otimes f^*$ for
morphisms $f,g$ in~$\cc$.

\subsection{Traces and dimensions}\label{sect-trace} For an endomorphism $f $ of an object $X$ of   a pivotal category $\cc$, one defines the
\emph{left} and \emph{right traces} $\tr_l(f), \tr_r(f) \in
\End_\cc(\un)$ by
$$\tr_l(f)=\lev_X(\id_{\ldual{X}} \otimes f) \rcoev_X  \quad {\text
{and}}\quad
 \tr_r(f)=\rev_X( f \otimes \id_{\ldual{X}}) \lcoev_X  .
$$
They satisfy $\tr_l(gh)=\tr_l(hg)$ and $
\tr_r(gh)=\tr_r(hg)$ for any morphisms $g\co X \to Y$ and   $h\co Y
\to X$  in $\cc$. Also we have  $\tr_l(f)=\tr_r(\ldual{f})=\tr_l(f^{**})$ for
any endomorphism $f $ in~$\cc$. If
\begin{equation}\label{special}
\alpha\otimes \id_X= \id_X\otimes \alpha \quad \text{for all $\alpha\in
\End_{\cc}(\un)$ and $X$ in $\cc$,}
\end{equation}
then $\tr_l,\tr_r$
are $\otimes$-multiplicative, that is,  $ \tr_l(f\otimes g)=\tr_l(f) \,
\tr_l(g)$ and $\tr_r(f\otimes g)=\tr_r(f) \, \tr_r(g) $ for all
endomorphisms $f,g$ in $\cc$.

The \emph{left} and the  \emph{right dimensions} of   an object $X$ of $\cc$
are defined by $ \dim_l(X)=\tr_l(\id_X) $ and $
\dim_r(X)=\tr_r(\id_X) $.  Isomorphic objects  have the same
dimensions, $\dim_l(X)=\dim_r(\ldual{X})=\dim_l(X^{**})$, and $\dim_l(\un)=\dim_r(\un)=\id_{\un}$. If~$\cc$ satisfies \eqref{special}, then  left and right dimensions are $\otimes$-multiplicative: $\dim_l (X\otimes Y)= \dim_l (X)
\dim_l (Y)$ and $\dim_r (X\otimes Y)= \dim_r (X) \dim_r (Y)$ for any
$X,Y$ in $\cc$.

\subsection{Penrose graphical calculus} We represent morphisms in a category $\cc$ by plane   diagrams to be read from the bottom to the top. In a pivotal category $\cc$, the  diagrams are made of   oriented arcs colored by objects of
$\cc$  and of boxes colored by morphisms of~$\cc$.  The arcs connect
the boxes and   have no mutual intersections or self-intersections.
The identity $\id_X$ of an object $X$ of $\cc$, a morphism $f\co X \to Y$,
and the composition of two morphisms $f\co X \to Y$ and $g\co Y \to
Z$ are represented as follows:
\begin{center}
\psfrag{X}[Bc][Bc]{\scalebox{.7}{$X$}} \psfrag{Y}[Bc][Bc]{\scalebox{.7}{$Y$}} \psfrag{h}[Bc][Bc]{\scalebox{.8}{$f$}} \psfrag{g}[Bc][Bc]{\scalebox{.8}{$g$}}
\psfrag{Z}[Bc][Bc]{\scalebox{.7}{$Z$}} $\id_X=$ \rsdraw{.45}{.9}{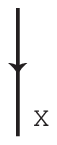}\,,\quad $f=$ \rsdraw{.45}{.9}{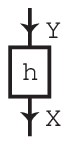} ,\quad \text{and} \quad $gf=$ \rsdraw{.45}{.9}{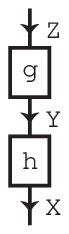}\,.
\end{center}
The monoidal product of two morphisms $f\co X \to Y$
and $g \co U \to V$ is represented by juxtaposition:
\begin{center}
\psfrag{X}[Bc][Bc]{\scalebox{.7}{$X$}} \psfrag{h}[Bc][Bc]{\scalebox{.8}{$f$}}
\psfrag{Y}[Bc][Bc]{\scalebox{.7}{$Y$}}  $f\otimes g=$ \rsdraw{.45}{.9}{morphism} \psfrag{X}[Bc][Bc]{\scalebox{.8}{$U$}} \psfrag{g}[Bc][Bc]{\scalebox{.8}{$g$}}
\psfrag{Y}[Bc][Bc]{\scalebox{.7}{$V$}} \rsdraw{.45}{.9}{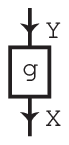}\,.
\end{center}
If an arc colored by $X$ is oriented upwards,
then the corresponding object   in the source/target of  morphisms
is $X^*$. For example, $\id_{X^*}$  and a morphism $f\co X^* \otimes
Y \to U \otimes V^* \otimes W$  may be depicted as:
\begin{center}
 $\id_{X^*}=$ \, \psfrag{X}[Bl][Bl]{\scalebox{.7}{$X$}}
\rsdraw{.45}{.9}{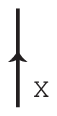} $=$  \,
\psfrag{X}[Bl][Bl]{\scalebox{.7}{$\ldual{X}$}}
\rsdraw{.45}{.9}{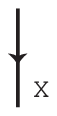}  \quad and \quad
\psfrag{X}[Bc][Bc]{\scalebox{.7}{$X$}}
\psfrag{h}[Bc][Bc]{\scalebox{.8}{$f$}}
\psfrag{Y}[Bc][Bc]{\scalebox{.7}{$Y$}}
\psfrag{U}[Bc][Bc]{\scalebox{.7}{$U$}}
\psfrag{V}[Bc][Bc]{\scalebox{.7}{$V$}}
\psfrag{W}[Bc][Bc]{\scalebox{.7}{$W$}} $f=$
\rsdraw{.45}{.9}{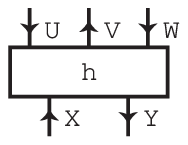} \,.
\end{center}
The duality morphisms   are depicted as follows:
\begin{center}
\psfrag{X}[Bc][Bc]{\scalebox{.7}{$X$}} $\lev_X=$ \rsdraw{.45}{.9}{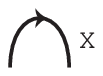}\,,\quad
 $\lcoev_X=$ \rsdraw{.45}{.9}{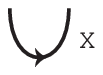}\,,\quad
$\rev_X=$ \rsdraw{.45}{.9}{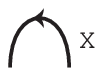}\,,\quad
\psfrag{C}[Bc][Bc]{\scalebox{.7}{$X$}} $\rcoev_X=$
\rsdraw{.45}{.9}{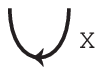}\,.
\end{center}
The dual of a morphism $f\co X \to Y$ and the
  traces of a morphism $g\co X \to X$ can be depicted as
follows:
\begin{center}
\psfrag{X}[Bc][Bc]{\scalebox{.7}{$X$}} \psfrag{h}[Bc][Bc]{\scalebox{.8}{$f$}}
\psfrag{Y}[Bc][Bc]{\scalebox{.7}{$Y$}} \psfrag{g}[Bc][Bc]{\scalebox{.8}{$g$}}
$f^*=$ \rsdraw{.45}{.9}{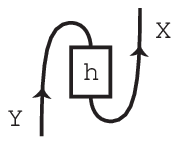}$=$ \rsdraw{.45}{.9}{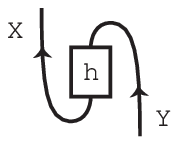}\quad \text{and} \quad
$\tr_l(g)=$ \rsdraw{.45}{.9}{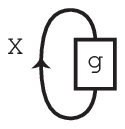}\,,\quad  $\tr_r(g)=$ \rsdraw{.45}{.9}{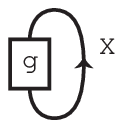}\,.
\end{center}
  In a pivotal category, the morphisms represented by the diagrams
are invariant under isotopies of the diagrams in the plane keeping
fixed the bottom and   top endpoints.

\subsection{Spherical categories}\label{sect-spherical}
  A \emph{spherical category} is a pivotal category whose left and
right traces are equal, i.e.,  $\tr_l(g)=\tr_r(g)$ for every
endomorphism $g$ of an object. Then $\tr_l(g)$ and $ \tr_r(g)$ are
denoted $\tr(g)$ and called the \emph{trace of $g$}. Similarly, the
left and right dimensions of an object~$X$ are denoted   $\dim(X)$
and called the \emph{dimension of $X$}.

Note that sphericity can be interpreted in graphical terms: it means that
the morphisms represented by
closed diagrams are invariant under isotopies of   diagrams in the 2-sphere $S^2=\RR^2\cup
\{\infty\}$, i.e., are preserved under isotopies pushing   arcs of
the diagrams across~$\infty$.


\subsection{Additive categories}
Let $\kk$ be a commutative ring. A \emph{\kt  additive category} is a category where
$\Hom$-sets are \kt modules,  the  composition of morphisms is \kt bilinear, and any finite family of objects has a direct sum. In particular, such a category has a zero object.

An object $X$ of a $\kk$-additive category $\cc$ is \emph{scalar} if
the map $\kk \to \End_\cc(X)$, $\alpha \mapsto \alpha\,\id_X$ is bijective.

A  \emph{$\kk$-additive monoidal category} is a monoidal category which is $\kk$-additive in such a way that the monoidal product
is \kt bilinear. Note that a $\kk$-additive monoidal category whose unit object $\un$ is scalar satisfies \eqref{special} and so  its traces $\tr_l,\tr_r$ are  \kt linear and $\otimes$-multiplicative.

\subsection{Fusion categories}\label{sphesphe}
A \emph{fusion category} over a commutative ring $\kk$ is a \kt additive rigid category $\cc$ such that
\begin{enumerate}
\labela
\item each object of $\cc$ is a finite direct sum of scalar objects;
\item for any non-isomorphic scalar objects $i,j  $ of $\cc$, we have $\Hom_\cc(i,j)=0$;
\item the set of isomorphism classes of scalar objects of~$\cc$   is finite;
\item the unit object $\un$ is scalar.
\end{enumerate}

Let $\cc$ be a fusion category.  The $\Hom$ spaces in
$\cc$ are free \kt modules of finite rank. We identify $\End_\cc(\un)$ with $\kk$ via the canonical isomorphism.
Given a scalar object $i$ of $\cc$, the \emph{$i$-isotypical component} $X^{(i)}$
 of an object $X$ is the largest direct factor of $X$ isomorphic to a direct sum of copies of $i$. The actual number of copies of $i$ is $$\nu_i(X)=\rank_\kk \,\Hom_\cc(i,X)=\rank_\kk \,\Hom_\cc(X,i).$$ An \emph{$i$-decomposition} of $X$ is an explicit direct sum decomposition of $X^{(i)}$ into copies of~$i$, that is, a family $(p_\alpha\co X \to i, q_\alpha \co i \to X)_{\alpha \in A}$ of pairs of morphisms  in $\cc$ such that
\begin{enumerate}
\labela
\item $p_\alpha \,q_\beta = \delta_{\alpha,\beta}\,\id_i$ for all $\alpha,\beta \in A$,
\item the set $A$ has $\nu_i(X)$ elements,
\end{enumerate}
where $\delta_{\alpha,\beta}$ is the Kronecker symbol.

A \emph{representative set of scalar objects} of  $\cc$ is a set~$I$ of scalar objects such that $\un \in I$ and every
scalar object of $\cc$ is isomorphic to exactly one element of~$I$.

Note that if $\kk$ is a field, a fusion category over $\kk$ is abelian  and semisimple.  Recall that an abelian category is \emph{semisimple} if its objects are direct sums of simple\footnote{An object of an abelian category is \emph{simple} if it is non-zero and has no other subobject than the zero object and itself.} objects.

A pivotal fusion category is spherical (see Section~\ref{sect-spherical}) if and only if the left and right dimension of any of its scalar objects coincide.

\subsection{Graphical calculus in pivotal fusion categories}
Let $\cc$ be a pivotal fusion category. Let $X$ be an object of $\cc$ and $i$ be a scalar object of $\cc$. Then the tensor
\begin{equation*}
\sum_{\alpha \in A} p_\alpha \otimes_\kk q_\alpha \in \Hom_\cc(X,i) \otimes_\kk \Hom_\cc(i,X),
\end{equation*}
where $( p_\alpha  , q_\alpha )_{\alpha \in A}$ is an $i$-decomposition of $X$, does not depend on the choice of the $i$-decomposition $( p_\alpha  , q_\alpha )_{\alpha \in A}$ of $X$.
Consequently, a sum of the type
\begin{equation*}
\psfrag{p}[Bc][Bc]{\scalebox{1}{$p_\alpha$}}
\psfrag{q}[Bc][Bc]{\scalebox{1}{$q_\alpha$}}
\psfrag{X}[Bl][Bl]{\scalebox{.9}{$X$}}
\psfrag{i}[Bl][Bl]{\scalebox{.9}{$i$}}
\sum_{\alpha \in A} \; \rsdraw{.45}{.9}{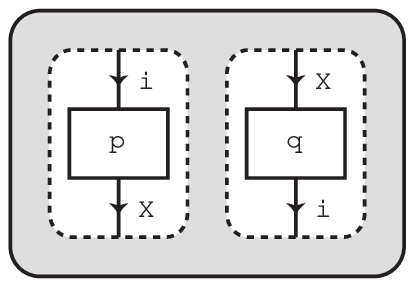}\,,
\end{equation*}
where $( p_\alpha  , q_\alpha )_{\alpha \in A}$ is an $i$-decomposition of an object $X$ and the gray area does not involve $\alpha$, represents a morphism in $\cc$ which is independent of the choice of the $i$-decomposition. We depict it as
\begin{equation}\label{eq-depict-tensor}
\psfrag{X}[Bl][Bl]{\scalebox{.9}{$X$}}
\psfrag{i}[Bl][Bl]{\scalebox{.9}{$i$}}
\phantom{\sum_{\alpha \in A}} \;\rsdraw{.45}{.9}{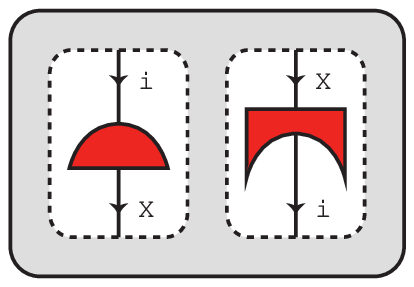}\,,
\end{equation}
where the two   curvilinear boxes should be shaded with the same
color. If several such pairs of boxes appear in a picture, they must
have different colors. We will also depict 
\begin{equation*}
\psfrag{e}[Br][Br]{\scalebox{.9}{$i$}}
\psfrag{R}[Br][Br]{\scalebox{.9}{$X$}}
\psfrag{X}[Bl][Bl]{\scalebox{.9}{$X$}}
\psfrag{i}[Bl][Bl]{\scalebox{.9}{$i$}}
\rsdraw{.45}{.9}{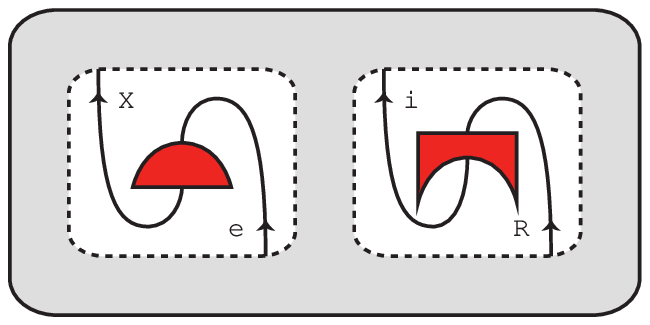} \quad \text{as} \quad \rsdraw{.45}{.9}{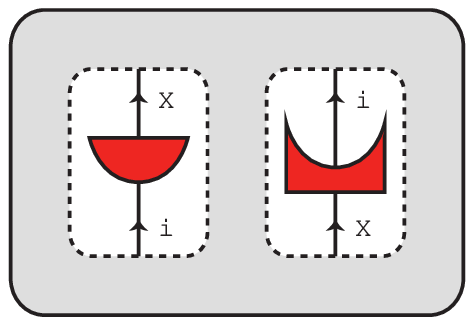}\,.
\end{equation*}
As usual, the edges labeled with  $i=\un$ may be erased and then \eqref{eq-depict-tensor} becomes
\begin{equation*}
\psfrag{X}[Bl][Bl]{\scalebox{.9}{$X$}}
\psfrag{i}[Bl][Bl]{\scalebox{.9}{$i$}}
\phantom{\sum_{\alpha \in A}} \;\rsdraw{.45}{.9}{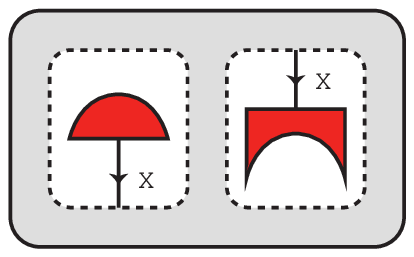}\,.
\end{equation*}
Note also that   tensor products of objects may be depicted as
bunches of   strands. For example,
\begin{equation*}
\psfrag{Z}[Bl][Bl]{\scalebox{.9}{$X^* \otimes Y \otimes Z^*$}}
\psfrag{Y}[Bl][Bl]{\scalebox{.9}{$Y$}}
\psfrag{T}[Bl][Bl]{\scalebox{.9}{$Z$}}
\psfrag{X}[Br][Br]{\scalebox{.9}{$X$}}
\psfrag{i}[Bl][Bl]{\scalebox{.9}{$i$}}
\rsdraw{.45}{.9}{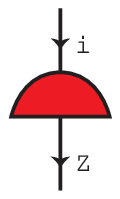} \qquad  \qquad \;= \; \rsdraw{.45}{.9}{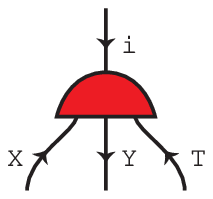} \qquad \, \text{and} \, \qquad
\rsdraw{.45}{.9}{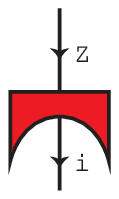} \qquad \qquad \; = \; \rsdraw{.45}{.9}{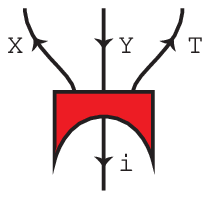}\;
\end{equation*}
where the equality sign means that the pictures represent the same
morphism of $\cc$.

\subsection{Braided and ribbon categories}

 A \emph{braiding} in a monoidal category $\bb$ is a natural isomorphism
$ \tau=\{\tau_{X,Y} \co  X \otimes Y\to Y \otimes X\}_{X,Y \in
\bb} $ such that
\begin{equation*}
\tau_{X, Y\otimes Z}=(\id_Y \otimes \tau_{X, Z})(\tau_{X, Y} \otimes \id_Z) \quad \text{and} \quad
\tau_{X\otimes Y,Z}=(\tau_{X,Z} \otimes \id_Y)(\id_X \otimes \tau_{Y,Z})
\end{equation*}
for all $X,Y,Z$ objects of $\cc$. These conditions imply that $\tau_{X,\un}=\tau_{\un,X}=\id_X$.

A monoidal category endowed with a braiding is said
to be \emph{braided}. The braiding    and its inverse are depicted
as
$$
\psfrag{X}[Bc][Bc]{\scalebox{.8}{$X$}}
\psfrag{Y}[Bc][Bc]{\scalebox{.8}{$Y$}}
\tau_{X,Y}=\,\rsdraw{.45}{.9}{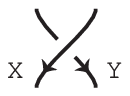} \quad \text{and} \quad
\tau^{-1}_{Y,X}=\,\rsdraw{.45}{.9}{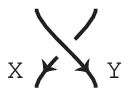} \, .
$$
Note that any braided category satisfies the condition \eqref{special} of Section~\ref{sect-trace}.

For any object $X$ of a braided pivotal category
$\bb$, the morphism
$$
\theta_X  =
\psfrag{X}[Br][Br]{\scalebox{.8}{$X$}}\rsdraw{.45}{.9}{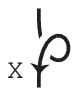}\,=(\id_X
\otimes \rev_X)(\tau_{X,X} \otimes \id_{X^*})(\id_X \otimes
\lcoev_X) \co X\to X,
$$
is called the \emph{twist}. The twist is natural in $X$ and  invertible, with inverse
$$
\theta_X^{-1}=\psfrag{X}[Bl][Bl]{\scalebox{.8}{$X$}}\rsdraw{.45}{.9}{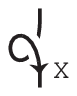}=(\lev_X \otimes \id_X)(\id_{X^*} \otimes \tau_{X,X}^{-1})(\rcoev_X \otimes \id_X)\co X \to X.
$$
It satisfies
$\theta_{X\otimes Y}=(\theta_X \otimes \theta_Y)\tau_{Y,X}\tau_{X,Y}$ for  all objects  $X,Y$ of $\bb$ and
 $\theta_\un=\id_\un$.

A \emph{ribbon category} is a braided pivotal category $\bb$ whose
twist $\theta$ is self-dual, i.e., $(\theta_X)^*=\theta_{X^*}$ for any object $X$ of $\bb$. This  is equivalent to the equality:
$$
\psfrag{X}[Br][Br]{\scalebox{.8}{$X$}} \rsdraw{.45}{.9}{theta1}
=\psfrag{X}[Bl][Bl]{\scalebox{.8}{$X$}}\rsdraw{.45}{.9}{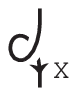}\,.
$$
A ribbon category  is spherical.

\subsection{The center of a  monoidal category}\label{sect-centerusual}
Let $\cc$ be a monoidal category. A \emph{half braiding} of $\cc$ is
a pair $({{A}},\sigma)$, where ${{A}}$ is an object of $\cc$ and
\begin{equation*}
\sigma=\{\sigma_X \co  {{A}} \otimes X\to X \otimes {{A}}\}_{X \in \cc}
\end{equation*}
is a natural isomorphism such that
 \begin{equation}\label{axiom-half-braiding}  \sigma_{X \otimes Y}=(\id_X \otimes
\sigma_Y)(\sigma_X \otimes \id_Y)\end{equation} for all
$X,Y$ objects of $\cc$. This implies that $\sigma_\un=\id_{{A}}$.

The \emph{center of $\cc$} is the braided category $\zz(\cc)$
defined as follows. The objects of~$\zz(\cc)$ are half braidings of
$\cc$. A morphism $({{A}},\sigma)\to ({{A}}',\sigma')$ in $\zz(\cc)$ is a
morphism $f \co {{A}} \to {{A}}'$ in $\cc$ such that $(\id_X \otimes
f)\sigma_X=\sigma'_X(f \otimes \id_X)$ for any object $X$ of $\cc$. The
  unit object of $\zz(\cc)$ is $\un_{\zz(\cc)}=(\un,\{\id_X\}_{X \in
\cc})$ and the monoidal product is
\begin{equation*}
({{A}},\sigma) \otimes ({{B}}, \rho)=\bigl({{A}} \otimes {{B}},(\sigma \otimes \id_{{B}})(\id_{{A}} \otimes \rho) \bigr).
\end{equation*}
The braiding $\tau$ in $\zz(\cc)$ is defined by
$$\tau_{({{A}},\sigma),({{B}}, \rho)}=\sigma_{{{B}}}\co ({{A}},\sigma) \otimes
({{B}}, \rho) \to ({{B}}, \rho) \otimes ({{A}},\sigma).$$

There is a  {\it forgetful functor} $\uu\co\zz(\cc)\to \cc$ assigning to
every half braiding $({{A}}, \sigma)$ the underlying object ${{A}}$ and
acting in the obvious way on the morphisms. This is a strict monoidal functor.

If $\cc$ satisfies \eqref{special}, then $\End_{\zz(\cc)}(\un_{\zz(\cc)})=\End_\cc(\un)$.

If $\cc$ is rigid, then  so is $\zz(\cc)$. If $\cc$ is pivotal, then so is $\zz(\cc)$ with
$({{A}},\sigma)^*=({{A}}^*,\sigma^\natural)$, where
$$
 \psfrag{M}[Bc][Bc]{\scalebox{.9}{${{A}}$}}
 \psfrag{X}[Bc][Bc]{\scalebox{.9}{$X$}}
 \psfrag{s}[Bc][Bc]{\scalebox{.9}{$\sigma_{X^*}$}}
\sigma^\natural_X= \rsdraw{.45}{1}{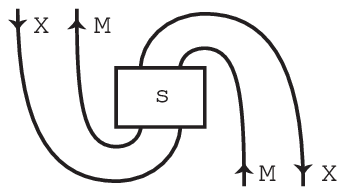} \co {{A}}^* \otimes X \to X \otimes {{A}}^*,
$$
and $\lev_{({{A}},\sigma)}= \lev_{{A}}$, $\lcoev_{({{A}},\sigma)}= \lcoev_{{A}}$,
$\rev_{({{A}},\sigma)}= \rev_{{A}}$, $\rcoev_{({{A}},\sigma)}= \rcoev_{{A}}$.
In that case the forgetful functor $\uu$ preserves (left and right) traces of morphisms and dimensions of objects.

If $\cc$ is a \kt additive monoidal category, then so is $\zz(\cc)$ and
the forgetful functor is \kt linear. If $\cc$ is  an abelian rigid category, then so is $\zz(\cc)$, and the forgetful functor  is exact.

If $\cc$ is a fusion category over the ring $\kk$, then $\zz(\cc)$ is braided \kt additive rigid category whose monoidal unit is scalar. If in addition $\kk$ is field,
then $\cc$ is abelian, and so is $\zz(\cc)$.

\section{Main results}\label{main-results}
In this section, we state our main results concerning the center of a pivotal fusion category. They are proved in Section~\ref{sect-proofs}. Let $\cc$ be a pivotal fusion category over a commutative ring $\kk$  and $I$ be a representative set of scalar objects of $\cc$. Recall from Section~\ref{sect-centerusual} that the center $\zz(\cc)$ of $\cc$ is a braided \kt additive pivotal category whose monoidal unit is scalar.

\begin{figure}[t]
   \begin{center}
\subfigure[The coproduct $\Delta\co C \to C \otimes C$]{
$\displaystyle
\Delta=
 \psfrag{i}[Bl][Bl]{\scalebox{.85}{$i$}}
 \psfrag{j}[Bl][Bl]{\scalebox{.85}{$j$}}
 \psfrag{k}[Bl][Bl]{\scalebox{.85}{$k$}}
 \psfrag{l}[Bl][Bl]{\scalebox{.85}{$\ell$}}
 \psfrag{n}[Bl][Bl]{\scalebox{.85}{$n$}}
 \!\!\!\sum_{i,j,k,\ell,n\in I} \;\; \rsdraw{.50}{.9}{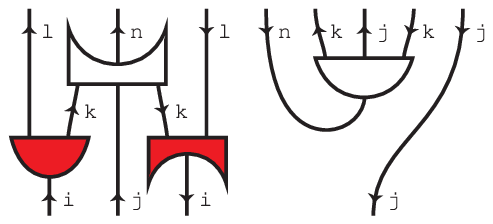}
$}\\
\subfigure[The product $m\co C \otimes C \to C$]{
$\displaystyle
m=
 \psfrag{i}[Bl][Bl]{\scalebox{.85}{$i$}}
 \psfrag{j}[Bl][Bl]{\scalebox{.85}{$j$}}
 \psfrag{k}[Bl][Bl]{\scalebox{.85}{$k$}}
 \psfrag{l}[Bl][Bl]{\scalebox{.85}{$\ell$}}
 \psfrag{c}[Br][Br]{\scalebox{.85}{$a$}}
 \psfrag{b}[Br][Br]{\scalebox{.85}{$\ell$}}
 \psfrag{a}[Bl][Bl]{\scalebox{.85}{$n$}}
 \!\!\sum_{i,j,k,\ell,n,a\in I} \;\;\; \rsdraw{.50}{.9}{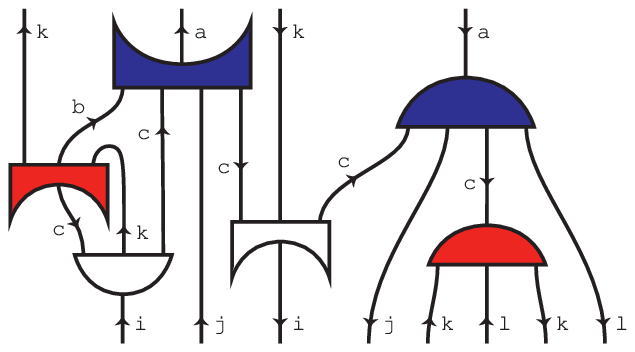}
$}\\
\setlength{\subfigcapskip}{5pt}
\subfigure[The counit $\varepsilon\co C \to \un$]{\phantom{XX}
$\displaystyle
\varepsilon=
 \psfrag{j}[Bl][Bl]{\scalebox{.85}{$j$}}
 \sum_{j\in I} \; \rsdraw{.10}{.9}{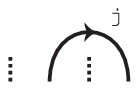}
$\phantom{XX}}\quad
\subfigure[The unit $u\co \un \to C$]{\phantom{XX}
$\displaystyle
u=
 \psfrag{i}[Bl][Bl]{\scalebox{.85}{$i$}}
 \sum_{i\in I} \; \;\rsdraw{.5}{.9}{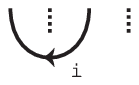}
$\phantom{XX}}\\
\setlength{\subfigcapskip}{10pt}
\subfigure[The antipode $S\co C \to C$]{
$\displaystyle
S=
 \psfrag{i}[Bl][Bl]{\scalebox{.85}{$i$}}
 \psfrag{j}[Bl][Bl]{\scalebox{.85}{$j$}}
 \psfrag{k}[Bl][Bl]{\scalebox{.85}{$k$}}
 \psfrag{l}[Bl][Bl]{\scalebox{.85}{$\ell$}}
 \psfrag{n}[Bl][Bl]{\scalebox{.85}{$n$}}
 \!\!\!\sum_{i,j,k,\ell,n\in I} \; \;\rsdraw{.50}{.9}{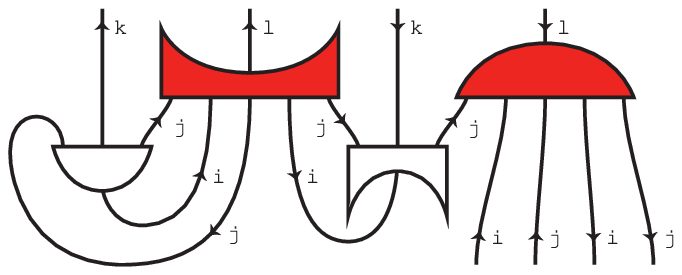}
$}\\
\subfigure[The canonical pairing $\omega\co C \otimes C \to \un$]{
$\displaystyle
\omega=
 \psfrag{i}[Bl][Bl]{\scalebox{.85}{$i$}}
 \psfrag{j}[Bl][Bl]{\scalebox{.85}{$j$}}
 \psfrag{k}[Bl][Bl]{\scalebox{.85}{$k$}}
 \psfrag{l}[Bl][Bl]{\scalebox{.85}{$\ell$}}
 \psfrag{n}[Bl][Bl]{\scalebox{.85}{$n$}}
 \!\sum_{i,j,k,\ell\in I} \;\; \rsdraw{.50}{.9}{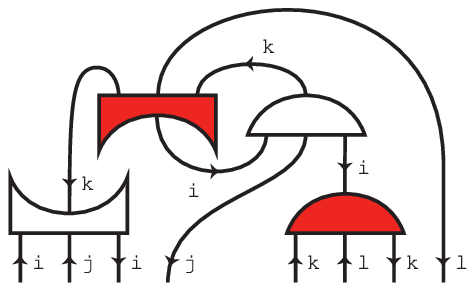}
$}
   \end{center}
   \caption{Structural morphisms of the coend of $\zz(\cc)$}
     \label{fig-struct-coend}
\end{figure}

The coend of a rigid braided category is, if it exists, a Hopf algebra in the category which coacts universally on the objects (see  Section~\ref{sect-coend-category} for details). The center $\zz(\cc)$ of $\cc$ has a coend $(C,\sigma)$, where
$$
C=\bigoplus_{i,j \in I} i^* \otimes j^* \otimes i \otimes j
$$
and the half braiding $\sigma=\{\sigma_Y\}_{Y \in \cc}$ is given by
\begin{equation}\label{eq-half-braiding-coend-ZC}
\sigma_Y=\psfrag{i}[Bl][Bl]{\scalebox{.85}{$i$}}
 \psfrag{j}[Bl][Bl]{\scalebox{.85}{$j$}}
 \psfrag{k}[Bl][Bl]{\scalebox{.85}{$k$}}
 \psfrag{l}[Bl][Bl]{\scalebox{.85}{$\ell$}}
 \psfrag{n}[Bl][Bl]{\scalebox{.85}{$n$}}
 \psfrag{X}[Bl][Bl]{\scalebox{.85}{$Y$}}
 \!\!\!\sum_{i,j,k,\ell,n\in I}  \rsdraw{.50}{.9}{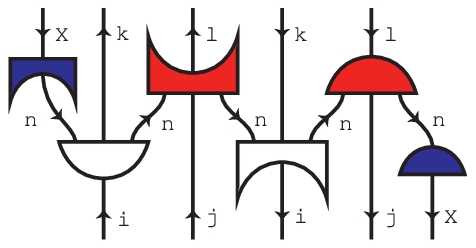} \; \co C \otimes Y \to Y \otimes C.
\end{equation}
The universal coaction  $\delta=\{\delta_{M,\gamma}\}_{(M,\gamma) \in \zz(\cc)}$ of the coend $(C,\sigma)$ is
\begin{equation}\label{eq-coaction-coend-ZC}
\delta_{(M,\gamma)}=
\psfrag{i}[Bl][Bl]{\scalebox{.85}{$i$}}
 \psfrag{j}[Bl][Bl]{\scalebox{.85}{$j$}}
 \psfrag{Y}[Br][Br]{\scalebox{.85}{$M$}}
 \psfrag{u}[cc][cc]{\scalebox{1}{$\gamma_i$}}
 \psfrag{X}[Bl][Bl]{\scalebox{.85}{$M$}}
\sum_{i,j\in I} \;\;\rsdraw{.45}{.9}{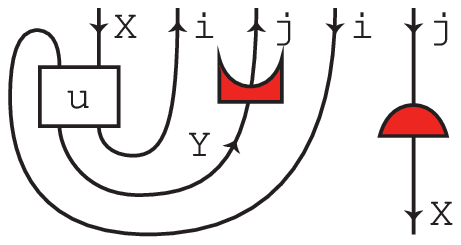}\;\co (M,\gamma) \to (M,\gamma) \otimes (C,\sigma).
\end{equation}
The  structural morphisms and the canonical pairing of the Hopf algebra $(C,\sigma)$ are depicted  in Figure~\ref{fig-struct-coend},  where the dotted lines in the picture represent $\id_\un$ and serve to indicate which direct factor of $C$ is concerned. Moreover
\begin{equation}\label{eq-def-Lambda}
\Lambda=\sum_{j \in I} \, \dim_r(j) \,  \psfrag{i}[Bl][Bl]{\scalebox{.85}{$j$}}
 \;\rsdraw{.5}{.9}{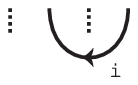} \; \co (\un,\id) \to (C,\sigma)
\end{equation}
is an integral of the Hopf algebra $(C,\sigma)$, which is invariant under the antipode.

By a modular category, we mean a braided pivotal category admitting a coend, and whose canonical pairing is non degenerate (see  Section~\ref{sect-modular-def} for details). The dimension of such a category is the dimension of its coend (see  Section~\ref{sect-dim-def}).

\begin{thm}\label{thm-center-modular}
The center $\zz(\cc)$ of $\cc$ is modular and has  dimension $\dim(\cc)^2$.
\end{thm}

The forgetful functor $\uu \co \zz(\cc) \to \cc$ has a left adjoint $\ff \co \cc \to \zz(\cc)$. For an object $X$ of $\cc$,
\begin{gather*}
\ff(X)=\bigl(Z(X),\varsigma_X=\{\varsigma_{X,Y}\}_{Y \in \cc}\bigr) \quad \text{where} \quad Z(X)=\bigoplus_{i \in I} i^* \otimes X \otimes i \quad \text{and}\\[-.7em]
 \varsigma_{X,Y}=\sum_{i,j \in I}\;\,
 \psfrag{i}[Bl][Bl]{\scalebox{.85}{$i$}}
 \psfrag{j}[Bl][Bl]{\scalebox{.85}{$j$}}
 \psfrag{Y}[Bl][Bl]{\scalebox{.85}{$Y$}}
 \psfrag{X}[Bl][Bl]{\scalebox{.85}{$X$}}
\rsdraw{.5}{.9}{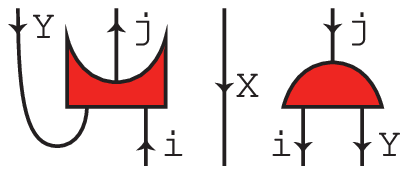}\, \co Z(X) \otimes Y \to Y \otimes Z(X).
\end{gather*}
For a morphism $f$ in $\cc$,
$$ \ff(f)=\sum_{i \in I} \id_{i^*} \otimes f \otimes \id_i.$$
By a \emph{free half braiding}, we mean an half braiding of the form $\ff(X)$ for some object $X$ of $\cc$.

\begin{thm}\label{thm-center-semisimple}
The dimension of $\cc$ is invertible in $\kk$ if and only if every half braiding is a retract of a free half braiding.
\end{thm}

From Section~\ref{sect-centerusual}, if $\kk$ is a field, then $\zz(\cc)$ is abelian.

\begin{cor}\label{cor-center-fusion}
Assume $\kk$ is a field. Then
\begin{enumerate}
\labela
\item The center $\zz(\cc)$ is semisimple (as an abelian category) if and only if $\dim(\cc) \neq 0$.
\item Assume $\kk$ is algebraically closed. Then $\zz(\cc)$
is a fusion category if and only if $\dim(\cc) \neq 0$.
\end{enumerate}
\end{cor}

Since the center of a spherical fusion category is ribbon (see, for example, \cite[Lemma 10.1]{TVi}), we recover M\"{u}ger's theorem:
\begin{cor}[{\cite[Theorem 1.2]{Mu}}]
If $\cc$ is a spherical fusion category over an algebraically closed field and $\dim (\cc)\neq 0$, then $\zz(\cc)$ is a modular ribbon fusion category (i.e., $\zz(\cc)$ is modular in the sense of \cite{Tu1}).
\end{cor}

Note that by  \cite{ENO}, the hypothesis $\dim (\cc)\neq 0$ of the previous corollary is automatically fulfilled  on a field of characteristic zero.

\begin{exa}
Let $G$ be a finite group and $\kk$ be a commutative ring. The category $\cc_{G,\kk}$ of $G$-graded free \kt modules of finite rank is a spherical fusion category. The dimension of $\cc_{G,\kk}$ is $\dim(\cc_{G,\kk})=|G| 1_\kk$, where $|G|$ is the order of $G$. By Theorem~\ref{thm-center-modular}, the center $\zz(\cc_{G,\kk})$ of $\cc_{G,\kk}$ is modular of dimension $|G|^2 1_\kk$. When $|G|$ is not invertible in $\kk$, by Theorem~\ref{thm-center-semisimple}, there exist  half braidings of $\cc_{G,\kk}$ which are not retracts of any free half braiding. If particular, if $\kk$ is a field of characteristic $p$ which divides $|G|$, then  $\zz(\cc_{G,\kk})$ is not semisimple.
\end{exa}

\section{Modular categories}\label{sect-Big-modular}

In this section, we clarify some notions used in the previous section.
More precisely, in Section~\ref{sect-HA}, we recall the definition of a Hopf algebra in a braided category and provide a criterion for the non-degeneracy of a Hopf algebra pairing. In Section~\ref{sect-coend}, we recall the definition of a coend. In Section~\ref{sect-coend-category}, we describe the Hopf algebra structure of the coend of a braided rigid category.
Sections~\ref{sect-dim-def} and~\ref{sect-modular-def} are devoted to the definition of respectively the dimension and the modularity of a braided category admitting a coend.

\subsection{Hopf algebras, pairings, and integrals}\label{sect-HA}
Let $\bb$ be a braided category, with braiding~$\tau$. Recall that a \emph{bialgebra in $\bb$} is an object $A$ of~$\bb$ endowed with four morphisms $m\co A \otimes A \to A$ (the product), $u\co \un \to A$ (the unit), $\Delta\co A \to A \otimes A$ (the coproduct), and $\varepsilon\co A \to \un$ (the counit) such that:
\begin{align*}
& m(m \otimes \id_A)=m(\id_A \otimes m), \qquad
 m(\id_A \otimes u)=\id_A=m(u \otimes \id_A), \\
& (\Delta \otimes \id_A)\Delta=(\id_A \otimes \Delta)\Delta,  \qquad (\id_A \otimes \varepsilon)\Delta=\id_A=(\varepsilon \otimes \id_A)\Delta,\\
& \Delta m=(m \otimes m)(\id_A \otimes \tau_{A,A} \otimes \id_A)(\Delta \otimes \Delta),\\
& \Delta u=u \otimes u, \qquad  \varepsilon m=\varepsilon \otimes \varepsilon, \qquad \varepsilon u=\id_\un.
\end{align*}
An \emph{antipode} for a bialgebra $A$ in $\bb$ is a morphism $S\co A \to A$ in $\bb$ such that
\begin{equation*}
m(S \otimes \id_A)\Delta=u \varepsilon=m(\id_A \otimes S)\Delta.
\end{equation*}
If it exists, an antipode is unique. A \emph{Hopf algebra in $\bb$} is a bialgebra in $\bb$ which admits an invertible antipode.

Let $A$ be a Hopf algebra in $\bb$. A \emph{Hopf pairing} for $A$ is a morphism $\omega\co A \otimes A \to \un$ such that
\begin{align*}
&\omega(m \otimes \id_A)=\omega (\id_A \otimes \omega \otimes \id_A)(\id_{A^{\otimes 2}} \otimes \Delta), && \omega(u
\otimes
\id_A)=\varepsilon,\\
&\omega(\id_A \otimes m)=\omega (\id_A \otimes \omega \otimes \id_A)(\Delta \otimes \id_{A^{\otimes 2}}), &&
\omega(\id_A \otimes u)=\varepsilon.
\end{align*}
These axioms imply that $\omega(S \otimes \id_A)=\omega(\id_A \otimes S)$.

A Hopf pairing $\omega$ for $A$ is \emph{non-degenerate} if there exists a morphism $\Omega \co \un \to A \otimes A$ in $\bb$ such that
$$
(\omega \otimes \id_A)(\id_A \otimes \Omega)=\id_A=(\id_A \otimes \omega)(\Omega \otimes \id_A).
$$
If such is the case, the morphism $\Omega$ is unique and called the \emph{inverse} of $\omega$.

A \emph{left} (resp.\@ \emph{right}) \emph{integral} for $A$ is a morphism $\Lambda \co \un \to A$ such that
$$
m(\id_A \otimes \Lambda)=\Lambda \, \varepsilon \qquad (\text{resp.\ } m(\Lambda \otimes \id_A)=\Lambda \, \varepsilon).
$$
A \emph{left} (resp.\@ \emph{right}) \emph{cointegral} for $A$ is a morphism $\lambda \co A \to \un $ such that
$$
(\id_A \otimes \lambda)\Delta=u\, \lambda \qquad (\text{resp.\ } (\lambda \otimes \id_A)\Delta=u\, \lambda).
$$
A (co)integral is \emph{two-sided} if it is both a left and a right (co)integral.

If $\Lambda$ is a left (resp.\@ right) integral for $A$, then $S\Lambda$ is a right (resp.\@ left) integral for $A$. If $\lambda$ is a left (resp.\@ right) cointegral for $A$, then $\lambda S$ is a right (resp.\@ left) cointegral for $A$.

Let $\omega$ be a Hopf pairing for $A$ and $\Lambda \co \un \to A$ be a morphism in $\bb$. Assume $\omega$ is non-degenerate. Then $\Lambda$ is a left  integral for $A$ if and only if $\lambda=\omega(\id_A \otimes \Lambda)$ is right cointegral for $A$, and $\Lambda$ is a right  integral for $A$ if and only if $\lambda=\omega(\Lambda \otimes \id_A)$ is left cointegral for $A$.

\begin{lem}\label{lem-non-degen}
Let $\omega$ be a Hopf pairing for a Hopf algebra $A$ in a braided category~$\bb$. Assume there exist morphisms $\Lambda, \Lambda' \co \un \to A$ in $\bb$ such that
\begin{enumerate}
  \renewcommand{\labelenumi}{{\rm (\alph{enumi})}}
\item  $\omega(\Lambda \otimes \id_A)$ and $\omega(\id_A \otimes \Lambda')$ are left cointegrals for $A$;
\item $\omega(\Lambda \otimes \Lambda')$ is invertible in $\End_\bb(\un)$.
\end{enumerate}
Then $\omega$ is non-degenerate, with inverse
$$
\Omega=\omega(\Lambda \otimes\Lambda')^{-1} \, (S\otimes \id_A \otimes \omega)(\id_A \otimes \Delta \Lambda \otimes \id_A) \Delta \Lambda',
$$
and $\Lambda$ and $\Lambda'$ are right integrals for $A$.
\end{lem}
\begin{proof}
Set $e = (S\otimes \id_A \otimes \omega)(\id_A \otimes \Delta \Lambda \otimes \id_A) \Delta \Lambda'\co \un \to A \otimes A$. Let us depict the product $m$, coproduct $\Delta$, antipode  $S$ of $A$, and the morphisms $\omega$, $\Lambda$, $\Lambda'$   as follows:
$$
 m= \,\rsdraw{.25}{.3}{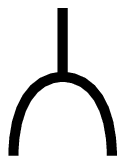}\;,\quad
 \Delta= \,\rsdraw{.25}{.3}{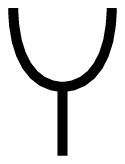}\;, \quad
 S = \, \rsdraw{.25}{.3}{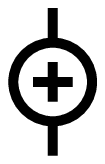}\;, \quad
 \omega = \, \rsdraw{.25}{.3}{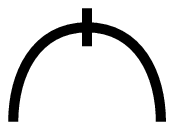}\;, \quad
 \Lambda = \, \rsdraw{.25}{.3}{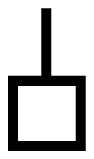}\;, \quad
 \Lambda' = \, \rsdraw{.25}{.3}{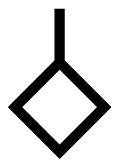}\;.
$$
Then
$(\id_A\otimes \omega) (e \otimes \id_A)=\omega(\Lambda \otimes \Lambda') \,\id_A$ since
$$
\rsdraw{.45}{.3}{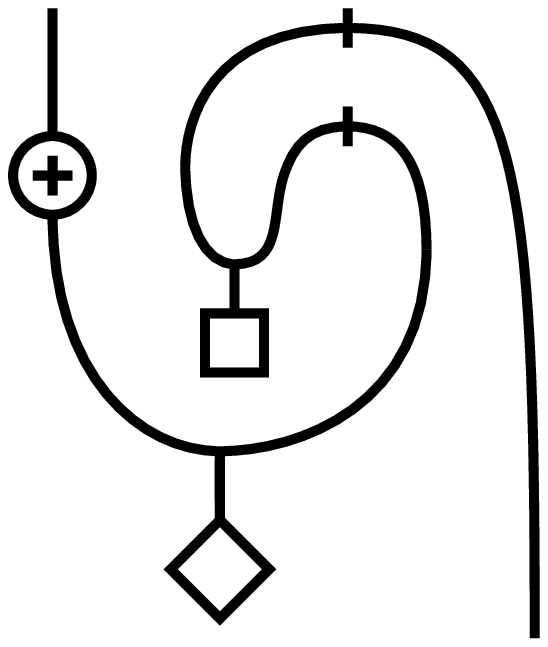}\;
 = \;\rsdraw{.45}{.3}{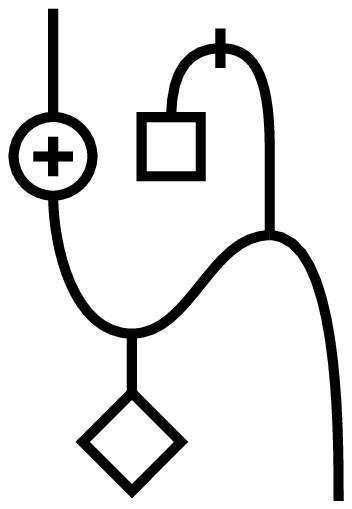} \;
 = \;\rsdraw{.45}{.3}{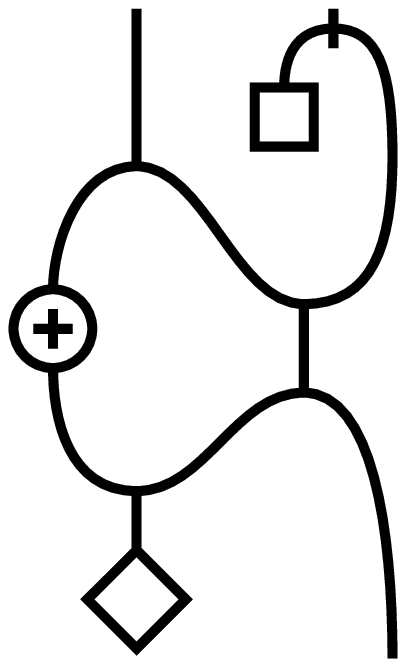}\;
 = \;\rsdraw{.45}{.3}{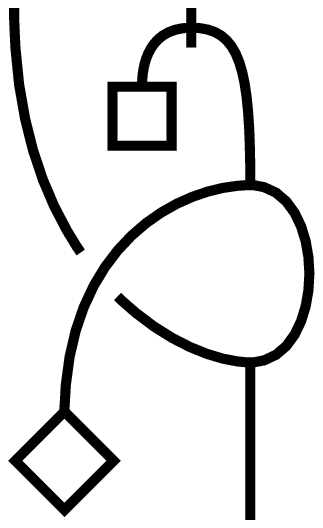}  \;
 = \;  \rsdraw{.45}{.3}{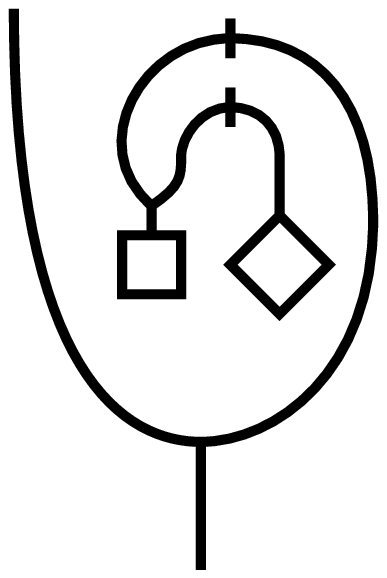}  \;
 = \;\rsdraw{.45}{.3}{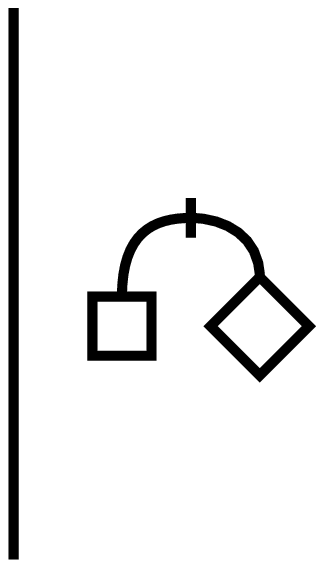} \;.
$$
We use the product/coproduct axioms of a Hopf pairing in the first and fourth  equalities, the unit axiom and the fact that  $\omega(\Lambda \otimes \id_A)$ is a left cointegral in the second equality, the compatibility of $m$ and $\Delta$ and the axiom of the antipode in the third equality, and finally the fact that $\omega(\id_A \otimes \Lambda')$ is a left cointegral and the unit/counit axiom of a Hopf pairing in the last equality.
Similarly one shows that $(\omega \otimes \id_A)(\id_A \otimes e) =  \omega(\Lambda \otimes \Lambda') \,\id_A$. Thus $\Omega=\omega(\Lambda \otimes\Lambda')^{-1} \, e$ is an inverse of $\omega$.

Finally, since  $\omega$ is non-degenerate and $\omega(\Lambda \otimes A)$ and  $\omega(A \otimes \Lambda')$ are left cointegrals, we conclude that $\Lambda$ and $\Lambda'$ are right integrals.
\end{proof}

\subsection{Coends}\label{sect-coend}
Let $\cc$ and $\dd$ be categories. A \emph{dinatural
transformation} from  a functor $F\co \dd^\opp \times \dd \to \cc$
to an object $A$ of $\cc$  is a family of morphisms in~$\cc$
$$
d=\{d_Y \co F(Y,Y) \to A\}_{Y \in \dd}
$$
such that for every morphism $f\co X \to Y$ in~$\dd$, we have
$$
d_X F(f, \id_X)=d_Y F(\id_Y,f)\colon F(Y,X)\to A.
$$
The {\it composition}  of such a $d$ with a morphism $\phi\co A\to B$ in $\cc$
is the dinatural transformation $ \phi\circ d= \{\phi \circ d_X \co
F(Y,Y) \to B\}_{Y \in \dd}$ from $F$ to $B$. A \emph{coend} of~$F$
is a pair $(C,\rho)$ consisting  in an object $C$ of $\cc$ and a
dinatural transformation $\rho$ from $F$ to $C$ satisfying the
following universality condition:   every dinatural transformation
$d$ from $F$ to  an object   of $\cc$ is the composition of $\rho$
with a morphism   in $\cc$  uniquely determined by $d$. If $F$ has a
coend $(C,\rho)$, then  it is unique (up to unique isomorphism). One
writes $C= \int^{Y \in \dd}F(Y,Y)$. For more on coends, see \cite{ML1}.

\begin{rem}\label{rem-coend-fusion}
Let $F\co \dd^\opp \times \dd \to \cc$ be \kt linear functor, where $\cc$ is a $\kk$-additive   category and $\dd$ is a fusion category (over $\kk$). Then $F$ has a coend. More precisely, pick  a  (finite)  representative set $I$ of simple objects of $\dd$
and  set $C=\oplus_{i \in I} F(i,i)$.  Let
  $\rho=\{\rho_Y \co F(Y,Y) \to C\}_{Y \in \dd}$ be defined by
$\rho_Y=\sum_{\alpha } F(q^\alpha_Y,p^\alpha_Y)$, where
$(p_Y^\alpha,q_Y^\alpha)_{\alpha  }$ is any $I$-partition of~$Y$. Then $(C, \rho)$ is a coend of $F$ and each
dinatural transformation $d$ from $F$ to any object $A$ of $\cc$ is the
composition of $\rho$ with    $ \oplus_{i \in I}\, d_i \co C \to A$.
\end{rem}

\subsection{The coend of a braided rigid category}\label{sect-coend-category}
Let $\bb$ be braided rigid category. The coend
\begin{equation*}
C=\int^{Y \in \bb} \leftidx{^\vee}{Y}{} \otimes Y,
\end{equation*}
if it exists, is called the \emph{coend of  $\bb$}.

Assume $\bb$ has a coend $C$ and denote by
$i_Y\co \leftidx{^\vee}{Y}{}  \otimes Y \to C$ the corresponding universal dinatural transformation. The
\emph{universal coaction} of $C$ on the objects of $\bb$ is the natural transformation $\delta$ defined by:
\begin{equation}\label{eq-coaction-def}
\delta_Y=(\id_Y \otimes i_Y)(\lcoev_Y \otimes \id_Y)\co Y \to Y \otimes C, \quad \text{depicted as} \quad \psfrag{C}[Bc][Bc]{\scalebox{.8}{$C$}} \psfrag{Y}[Bc][Bc]{\scalebox{.8}{$Y$}}\delta_Y=\rsdraw{.45}{.95}{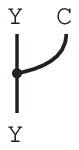}.
\end{equation}
According to Majid~\cite{Maj2}, $C$ is a Hopf algebra in $\bb$. Its coproduct $\Delta$, product $m$,  counit $\varepsilon$, unit $u$, and antipode $S$ with inverse $S^{-1}$ are characterized by the following equalities, where $X,Y\in\bb$:
\begin{gather*}
\psfrag{Y}[Bc][Bc]{\scalebox{.8}{$Y$}}
\psfrag{C}[Bc][Bc]{\scalebox{.8}{$C$}}
\psfrag{D}[cc][cc]{\scalebox{.9}{$\Delta$}}
\rsdraw{.45}{.9}{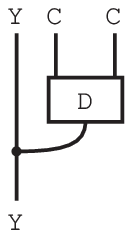} \, = \;\, \rsdraw{.45}{.9}{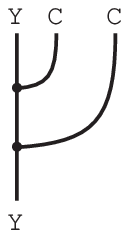} , \quad
\psfrag{Y}[Bc][Bc]{\scalebox{.8}{$Y$}}
\psfrag{C}[Bc][Bc]{\scalebox{.8}{$C$}}
\psfrag{D}[cc][cc]{\scalebox{1}{$\varepsilon$}}
\quad\quad\rsdraw{.45}{.9}{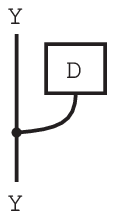} \; = \; \psfrag{Y}[Bc][Bc]{\scalebox{.8}{$Y$}}\rsdraw{.45}{.9}{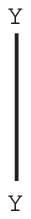} \; , \qquad \;\;
\psfrag{Y}[Bc][Bc]{\scalebox{.8}{$Y$}}
\psfrag{X}[Bc][Bc]{\scalebox{.8}{$X$}}
\psfrag{Z}[Bc][Bc]{\scalebox{.8}{$X\otimes Y$}}
\psfrag{C}[Bc][Bc]{\scalebox{.8}{$C$}}
\psfrag{m}[cc][cc]{\scalebox{.9}{$m$}}
\rsdraw{.45}{.9}{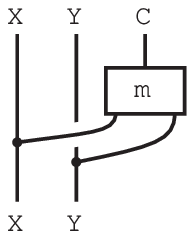} \, = \;\, \rsdraw{.45}{.9}{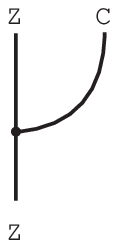} \,,\\
 u=\delta_\un, \qquad\;
\psfrag{a}[Bc][Bc]{\scalebox{.9}{$\lev_{Y}$}}
\psfrag{u}[Bc][Bc]{\scalebox{.9}{$\lcoev_{Y}$}}
\psfrag{Y}[Br][Br]{\scalebox{.8}{$Y$}}
\psfrag{C}[Bl][Bl]{\scalebox{.8}{$C$}}
\psfrag{D}[cc][cc]{\scalebox{.9}{$S$}} \rsdraw{.45}{.9}{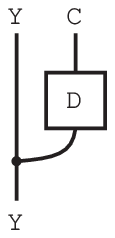} \; = \,
\rsdraw{.45}{.9}{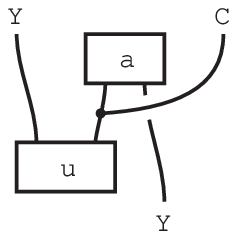}\, ,
\qquad
\psfrag{a}[Bc][Bc]{\scalebox{.9}{$\rev_{Y}$}}
\psfrag{u}[Bc][Bc]{\scalebox{.9}{$\rcoev_{Y}$}}
\psfrag{Y}[Br][Br]{\scalebox{.8}{$Y$}}
\psfrag{C}[Bl][Bl]{\scalebox{.8}{$C$}}
\psfrag{D}[cc][cc]{\scalebox{.9}{$S^{-1}$}} \rsdraw{.45}{.9}{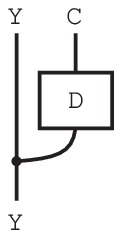} \; = \,
\rsdraw{.45}{.9}{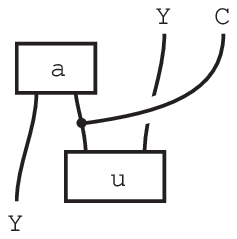}\,.
\end{gather*}
Furthermore, the morphism $\omega\co C \otimes C \to \un$ defined by
\begin{equation*}
\psfrag{Y}[Bl][Bl]{\scalebox{.8}{$Y$}}\psfrag{X}[Br][Br]{\scalebox{.8}{$X$}}
\psfrag{C}[Bl][Bl]{\scalebox{.8}{$C$}}
\psfrag{w}[cc][cc]{\scalebox{.9}{$\omega$}}
\rsdraw{.45}{.9}{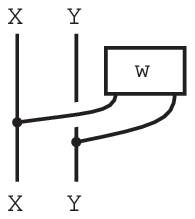} \, = \;\, \rsdraw{.45}{.9}{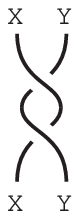}
\end{equation*}
is a Hopf pairing for $C$, called the \emph{canonical pairing}. Moreover this pairing satisfies the following self-duality condition:    $\omega \tau_{C,C} (S \otimes
S)=\omega$.

\subsection{The dimension of a braided pivotal category}\label{sect-dim-def}
Let $\bb$ be a braided pivotal category admitting a coend $C$.

\begin{lem}
The left and right dimension of $C$ coincide.
\end{lem}
\begin{proof}
Let $\upsilon=\{\upsilon_X\}_{X \in \bb}$ be the natural transformation defined by
$$
\upsilon_X=\;\psfrag{X}[Bc][Bc]{\scalebox{.8}{$X$}} \rsdraw{.45}{.9}{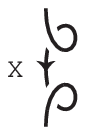} \; \co X \to X.
$$
Then $\upsilon$ is natural monoidal isomorphism,  that is, $\upsilon_{X \otimes Y} = \upsilon_X \otimes \upsilon_Y$ and $\upsilon_\un = \id_\un$,
which  implies that ${\upsilon_X}^* = \upsilon^{-1}_{X^*}$.
The full subcategory $\bb_0$  of $\bb$ made of the objects $X$ of $\bb$ satisfying $\tau_X = \id_X$ is a ribbon category. Let us prove that the coend $C$ of $\bb$ belongs to $\bb_0$. Denote by  $i=\{i_X\co X^* \otimes X \to C\}_{X \in \bb}$ the universal dinatural transformation associated with $C$.
For any object $X$ of $\cc$, by naturality and monoidality of $\upsilon$ and dinaturality of $i$, the following holds
$$
\upsilon_C i_X =i_X \upsilon_{(X^*\otimes X)}= i_X(\upsilon_{X^*} \otimes \upsilon_X) = i_X(\upsilon_X^*\upsilon_{X^*} \otimes \id_{X})=i_X.
$$
So $\upsilon_C=\id_C$, that is, $C$ belongs to $\bb_0$.  Hence the left and right dimension of $C$ coincide, since $\bb_0$ is a ribbon category.
\end{proof}

We define the \emph{dimension of $\bb$} as $\dim(\bb)=\dim_l(C)= \dim_r(C)$.

This definition agrees with the standard definition of  the dimension of a pivotal fusion category. Indeed, any pivotal fusion category $\cc$ (over the ring $\kk$) admits a coend $C=\oplus_{i \in I} i^* \otimes i$, where $I$ is a (finite)  representative set of scalar objects of $\cc$, and so
$$
\dim_l(C)=\dim_r(C)=\sum_{i\in I} \dim_l(i^*)\dim_l(i)=\sum_{i\in I} \dim_r(i)\dim_l(i).
$$

\subsection{Modular categories}\label{sect-modular-def}
By a \emph{modular category}, we mean a braided rigid category which admits a coend whose canonical pairing is non-degenerate.
Note that when $\bb$ is ribbon, this definition coincides with that of a \emph{2-modular category} given in \cite{Lyu2}.

\begin{rem}
Let $\bb$ be a braided pivotal fusion category over $\kk$. Let $I$ be a representative set of the scalar objects of $\bb$. Recall that  $C=\oplus_{i \in I} i^* \otimes i$ is the coend of $\bb$. For $i,j \in I$, set
$$
S_{i,j}=(\lev_i \otimes \rev_j)(\id_{i^*} \otimes \tau_{j,i}\tau_{i,j} \otimes \id_{j^*})(\rcoev_i \otimes \lcoev_j) \in \kk.
$$
The matrix $S=[S_{i,j}]_{i,j \in I}$, called the \emph{$S$-matrix} of $\bb$, is invertible if and only if the canonical pairing of $C$ is non-degenerate. In particular a modular category in the sense of~\cite{Tu1} is a ribbon fusion category which is modular in the above sense.
\end{rem}

\section{Proofs}\label{sect-proofs}
 The statements of  Section~\ref{main-results}  derive directly from the theory of Hopf monads,  introduced in
\cite{BV2} and developed it in  \cite{BV3,BLV}. Hopf monads generalize
Hopf algebras in the setting of general monoidal categories. In Section~\ref{sect-Hopf-monoads}, we recall some basic definitions  concerning Hopf monads.  In Section~\ref{sect-ZC-monadic}, we give a  Hopf monadic description of the center $\zz(\cc)$ of a fusion category $\cc$, from which is derived the explicit description of the coend of $\zz(\cc)$. In Section~\ref{sect-handleslide}, we prove a `handleslide'  property  for pivotal fusion categories. In Section~\ref{sect-proof-modular}, we use the explicit description of the coend of $\zz(\cc)$ to prove Theorem~\ref{thm-center-modular}  and prove that the morphism $\Lambda$ of \eqref{eq-def-Lambda} is an integral invariant under the antipode.  Sections~\ref{sect-proof-semisimple} and~\ref{sect-proof-coro} are devoted to the proofs of Theorem~\ref{thm-center-semisimple} and Corollary~\ref{cor-center-fusion}, respectively.

\subsection{Hopf monads and their modules}\label{sect-Hopf-monoads}
Let $\cc$ be a category.
A \emph{monad} on $\cc$  is a monoid in the category of endofunctors of $\cc$, that
is, a triple $(T,\mu,\eta)$ consisting of a functor $T\co \cc \to
\cc$ and two natural transformations $$\mu=\{\mu_X\co T^2(X) \to
T(X)\}_{X \in \cc}\quad {\text {and}} \quad  \eta=\{\eta_X\co X \to
T(X)\}_{X \in \cc}$$  called the \emph{product} and the \emph{unit}
of $T$, such that for any object $X$ of $\cc$, $$\mu_X
T(\mu_X)=\mu_X\mu_{T(X)} \quad {\text {and}} \quad
\mu_X\eta_{T(X)}=\id_{T(X)}=\mu_X T(\eta_X).$$
Given a monad $T=(T, \mu, \eta)$ on $\cc$, a  $T$\ti module in
$\cc$ is a pair $(M,r)$ where $M$ is an object of $\cc$ and $r\co T(M) \to M$ is a
morphism in $\cc$ such that $r T(r)= r \mu_M$ and $r \eta_M= \id_M$.
A morphism from a $T$\ti module  $(M,r)$ to a $T$\ti module $(N,s)$ is a morphism $f \co M \to N$ in $\cc$ such that $f r=s T(f)$.
This defines the {\it category  $\cc^T$  of $T$-modules in $\cc$} with composition induced by that in $\cc$. We
define a forgetful functor  $U_T\co\cc^T \to \cc$    by $U_T(M,r)=M$ and
$U_T(f)=f$.  The forgetful functor $U_T$ has a left adjoint $F_T \co \cc \to \cc^T$, called the free module functor, defined by $F_T(X)=(T(X),\mu_X)$  and $F_T(f)=T(f)$. Note that if $\cc$ is $\kk$-additive and $T$ is $\kk$-linear (that is, $T$
induces $\kk$-linear maps on Hom spaces), then the category $\cc^T$
is $\kk$-additive and the functors $U_T$ and $F_T$ are  $\kk$-linear.

Let $\cc$ be a monoidal category. A \emph{bimonad} on    $\cc$ is a monoid in the
category of  comonoidal endofunctors of $\cc$. In other words, a bimonad on $\cc$ is a
monad $(T,\mu,\eta)$ on $\cc$ such that the  functor $T\co \cc \to
\cc$ and the natural transformations $\mu$ and $\eta$ are
comonoidal, that is, $T$ comes equipped with a natural transformation $ T_2=\{T_2(X,Y) \co  T(X \otimes Y)\to T(X)
\otimes T(Y)\}_{X,Y \in \cc} $ and a morphism $T_0\co T(\un) \to \un$ such that
\begin{align*}
& \bigl(\id_{T(X)} \otimes T_2(Y,Z)\bigr) T_2(X,Y \otimes Z)= \bigl(T_2(X,Y) \otimes \id_{T(Z)}\bigr) T_2(X \otimes Y, Z) ;\\
& (\id_{T(X)} \otimes T_0) T_2(X,\un)=\id_{T(X)}=(T_0 \otimes \id_{T(X)}) T_2(\un,X) ;\\
& T_2(X,Y)\mu_{X \otimes Y}=(\mu_X \otimes \mu_Y) T_2(T(X),T(Y))T(T_2(X,Y));\\
& T_2(X,Y)\eta_{X \otimes Y}=\eta_X \otimes \eta_Y.
\end{align*}
For any bimonad $T$ on $\cc$, the category  of $T$\ti modules
$\cc^T$  has a monoidal structure with unit object $(\un,T_0)$ and
with tensor product
$$(M,r) \otimes (N,s)=\bigl(M \otimes N, (r \otimes s) \,
T_2(M,N)\bigr). $$ Note that the forgetful functor $U_T\co \cc^T \to
\cc$ is strict monoidal.

Given a  bimonad $(T,\mu,\eta)$ on $\cc$ and objects $X, Y\in \cc$,
one defines  the \emph{left fusion operator}
$$H^l_{X,Y} =(T(X) \otimes \mu_Y)T_2(X,T(Y)) \colon T(X\otimes T(Y)) \to T(X) \otimes
T(Y)$$
 and the \emph{right fusion operator}
$$
 H^r_{X,Y}=(\mu_X \otimes
T(Y))T_2(T(X),Y)\colon T(T(X)\otimes Y) \to T(X)\otimes T(Y).$$
A \emph{Hopf monad} on   $\cc$ is a bimonad on $\cc$
whose left and right fusion operators  are isomorphisms for all objects $X, Y$ of $\cc$.
When $\cc$ is a rigid category,    a bimonad $T$ on
$\cc$ is a Hopf monad if and only if the category  $\cc^T$ is rigid.
The  structure of a rigid category in $\cc^T$ can then be
 encoded in terms of natural transformations
$$s^l=\{s^l_X\co T(\leftidx{^\vee}{T}{}(X)) \to \leftidx{^\vee}{X}{}\}_{X \in \cc}
\quad {\text {and}} \quad s^r=\{s^r_X\co T(T(X)^\vee) \to
X^\vee\}_{X \in \cc}$$ called the \emph{left and right antipodes}.
They are computed from the   fusion operators:
\begin{align*}
& s^l_X= \bigl(T_0T(\lev_{T(X)})(H^l_{\leftidx{^\vee}{T}{}(X),X})^{-1} \otimes \leftidx{^\vee}{\eta}{_X}\bigr)
\bigl(\id_{T(\leftidx{^\vee}{T}{}(X))} \otimes \lcoev_{T(X)}\bigr);\\
& s^r_X= \bigl(\eta_X^\vee \otimes T_0T(\rev_{T(X)})(H^r_{X,T(X)^\vee})^{-1}\bigr)
\bigl(\rcoev_{T(X)} \otimes \id_{T(T(X)^\vee)}\bigr).
\end{align*}
The  left and right duals of any $T$\ti module $(M,r)$ are then defined
  by
$$
\leftidx{^\vee}{(}{} M,r)=(\leftidx{^\vee}{M}{}, s^l_M T(\leftidx{^\vee}{r}{})  \quad \text{and} \quad(M,r)^\vee=(M^\vee, s^r_M T(r^\vee).
$$

A \emph{quasitriangular Hopf monad} on $\cc$ is a Hopf monad $T$ on $\cc$  equipped with an \Rt matrix, that is, a natural transformation $$R=\{R_{X,Y}\co
X \otimes Y \to T(Y) \otimes T(X)\}_{X,Y \in \cc}$$ satisfying appropriate axioms which ensure
that  the natural transformation $ \tau=\{\tau_{(M,r),(N,s)}\}_{(M,r), (N,s) \in \cc^T}$ defined by
$$
\tau_{(M,r),(N,s)}=(s \otimes r) R_{M,N}\co (M,r) \otimes (N,s) \to (N,s) \otimes (M,r)
$$
form a braiding in the category $\cc^T$ of $T$-modules.

\subsection{The coend of the center of a fusion category}\label{sect-ZC-monadic}
Let $\cc$ be a pivotal fusion category (over the ring $\kk$), with  a representative set of scalar objects $I$.
 For each object $X$  of $\cc$, by Remark~\ref{rem-coend-fusion}, the \kt linear functor $\cc^\opp \times \cc \to \cc$, defined by $(U,V) \mapsto U^* \otimes X \otimes V$, has a coend
$$Z(X)=\bigoplus_{i \in I} i^* \otimes X \otimes i,$$
with dinatural transformation $\rho_X=\{\rho_{X,Y}\}_{Y \in \cc}$ given by
$$
\rho_{X,Y}=\sum_{i \in I}\;\,
 \psfrag{i}[Bl][Bl]{\scalebox{.85}{$i$}}
 \psfrag{X}[Bl][Bl]{\scalebox{.85}{$X$}}
 \psfrag{Y}[Bl][Bl]{\scalebox{.85}{$Y$}}
\rsdraw{.5}{.9}{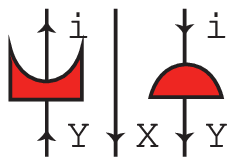}  \; \co Y^* \otimes X \otimes Y \to Z(X).
$$
The correspondence $X\mapsto Z(X)$  extends to a  functor $Z\co \cc\to \cc$. By Theorem  6.4 and Section 9.2 of \cite{BV3}, $Z$ is a quasitriangular Hopf monad on $\cc$ with the structural morphisms of $Z$ given in Figure~\ref{fig-strucZ-fusion}, where the dotted lines in the picture represent $\id_\un$.
\begin{figure}[t]
\begin{center}
$\displaystyle Z_2(X,Y)=\sum_{i \in I}$\,
 \psfrag{i}[Bc][Bc]{\scalebox{.85}{$i$}}
 \psfrag{X}[Bc][Bc]{\scalebox{.85}{$X$}}
 \psfrag{Y}[Bc][Bc]{\scalebox{.85}{$Y$}}
 \rsdraw{.45}{.9}{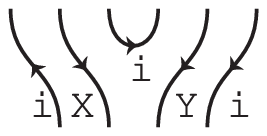}  $\co Z(X\otimes Y) \to Z(X)\otimes Z(Y)$, \\[1em]
$\displaystyle Z_0=\sum_{i \in I}$\,
 \psfrag{i}[Bc][Bc]{\scalebox{.85}{$i$}}
 \rsdraw{.25}{.9}{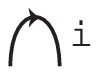} $\co Z(\un) \to \un$, \\[.3em]
$\displaystyle \mu_X=\!\!\sum_{i,j,k \in I}$\;
 \psfrag{i}[Bc][Bc]{\scalebox{.85}{$i$}}
 \psfrag{j}[Bc][Bc]{\scalebox{.85}{$j$}}
 \psfrag{k}[Bc][Bc]{\scalebox{.85}{$k$}}
 \psfrag{X}[Bc][Bc]{\scalebox{.85}{$X$}}
 \rsdraw{.5}{.9}{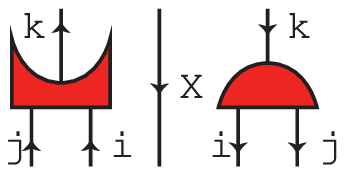}   $\co Z^2(X ) \to Z(X) $, \\[1em]
$\displaystyle \eta_X=\,$
  \psfrag{X}[Bl][Bl]{\scalebox{.85}{$X$}}
 \rsdraw{.5}{.9}{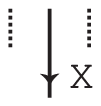} $\co X \to X=\un^* \otimes X \otimes \un \hookrightarrow Z(X)$, \\[.8em]
$\displaystyle s^l_X=s^r_X=\sum_{i,j \in I}$
 \psfrag{i}[Bc][Bc]{\scalebox{.85}{$i$}}
 \psfrag{u}[Bc][Bc]{\scalebox{.85}{$i^*$}}
 \psfrag{j}[Bc][Bc]{\scalebox{.85}{$j$}}
 \psfrag{X}[Bc][Bc]{\scalebox{.85}{$X$}}
 \,\rsdraw{.35}{.9}{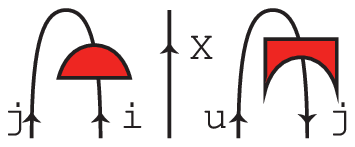} $\co Z(Z(X)^*) \to X^*$, \\[1em]
$\displaystyle  R_{X,Y}=\sum_{i \in I}$\;
 \psfrag{i}[Bl][Bl]{\scalebox{.85}{$i$}}
 \psfrag{X}[Br][Br]{\scalebox{.85}{$X$}}
 \psfrag{Y}[Bl][Bl]{\scalebox{.85}{$Y$}}
 \rsdraw{.5}{.9}{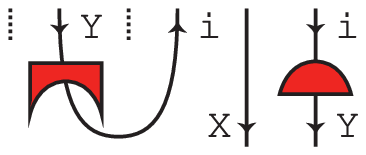}   $\co X \otimes Y \to Z(Y) \otimes Z(X) $.
\end{center}
\caption{Structural morphisms of the Hopf monad $Z$}
\label{fig-strucZ-fusion}
\end{figure}
In particular, the category $\cc^Z$ of $Z$-modules is a braided pivotal category. By \cite[Theorem 6.5]{BV3}, the functor
\begin{equation}\label{eq-iso-Phi}
\Phi\co \left \{\begin{array}{ccc}
\cc^Z & \to & \zz(\cc) \\
(M,r) & \mapsto & (M, \sigma)\\
f & \mapsto & f
\end{array} \right. \; \text{where} \quad \sigma_Y=\sum_{i \in I}\;\,
 \psfrag{i}[Bl][Bl]{\scalebox{.85}{$i$}}
 \psfrag{X}[Br][Br]{\scalebox{.85}{$M$}}
 \psfrag{Y}[Bl][Bl]{\scalebox{.85}{$Y$}}
 \psfrag{M}[Bl][Bl]{\scalebox{.85}{$M$}}
 \psfrag{r}[cc][cc]{\scalebox{1}{$r$}}
\rsdraw{.5}{.9}{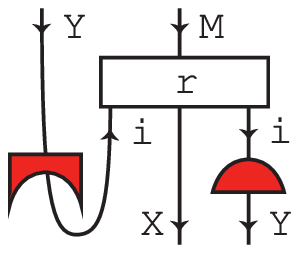}
\end{equation}
is an isomorphism of braided pivotal categories. Note that this isomorphism is a ``fusion'' version of the braided isomorphism $\zz(\mathrm{mod}_H) \simeq \mathrm{mod}_{D(H)}$ between the center of the category of modules over a finite dimensional Hopf algebra $H$ and the category of modules over the Drinfeld double $D(H)$ of $H$.
Now by \cite[Section 6.3]{BV3}, the coend of $\cc^Z$ is $(C,\alpha)$, where
$$
C=\bigoplus_{i,j \in I} i^* \otimes j^* \otimes i \otimes j \quad \text{and} \quad
\psfrag{i}[Bl][Bl]{\scalebox{.85}{$i$}}
 \psfrag{j}[Bl][Bl]{\scalebox{.85}{$j$}}
 \psfrag{k}[Bl][Bl]{\scalebox{.85}{$k$}}
 \psfrag{l}[Bl][Bl]{\scalebox{.85}{$l$}}
 \psfrag{n}[Bl][Bl]{\scalebox{.85}{$n$}}
\alpha=  \!\!\!\sum_{i,j,k,l,n\in I} \; \rsdraw{.50}{.9}{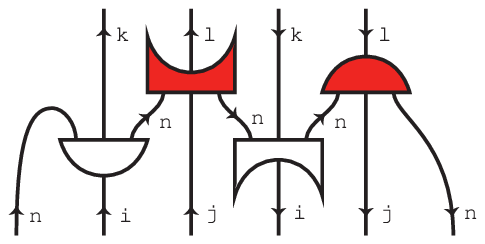},
$$

with universal dinatural transformation $\iota=\{\iota_{(M,r)}\}_{(M,r) \in \cc^Z}$ given  by
$$
\iota_{(M,r)}=
 \psfrag{i}[Bl][Bl]{\scalebox{.85}{$i$}}
 \psfrag{j}[Bl][Bl]{\scalebox{.85}{$j$}}
 \psfrag{X}[Bl][Bl]{\scalebox{.85}{$M$}}
 \psfrag{r}[cc][cc]{\scalebox{1}{$r^*$}}
\sum_{i,j \in I} \;\, \rsdraw{.50}{.9}{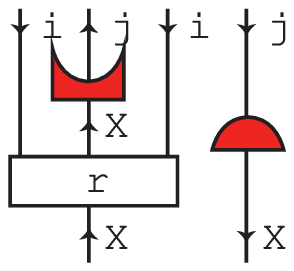} \; \co (M,r)^* \otimes (M,r) \to (C,\alpha).
$$
Thus $(C,\sigma)=\Phi(C,\alpha)$ is the coend of $\zz(\cc)$, with universal dinatural transformation $\{\Phi(\iota_{\Phi^{-1}(M,\gamma)})\}_{(M,\gamma)\in\zz(\cc)}$. Using the description of $\Phi$ and the definition of the universal coaction given in \eqref{eq-coaction-def}, we obtain that the half braiding $\sigma$ is given by \eqref{eq-half-braiding-coend-ZC} and  that the universal coaction of $(C,\sigma)$ is given by \eqref{eq-coaction-coend-ZC}.
Finally, recall from Section~\ref{sect-coend-category} that $(C,\alpha)$ is a Hopf algebra in $\cc^Z$ endowed with a canonical Hopf algebra pairing. By  \cite[Section 9.3]{BV3}, the structural morphisms of $(C,\alpha)$  are given in Figure~\ref{fig-struct-coend}.
These structural morphisms are also those of $(C,\sigma)$, since $\Phi$ is the identity on morphisms.

\subsection{Slope and handleslide in pivotal fusion categories}\label{sect-handleslide}
Let $\cc$ be a pivotal fusion category. Recall that the  left and right dimensions of a scalar object of $\cc$ are invertible.
The \emph{slope}  of a scalar object $i$
is the invertible scalar  $\slo(i)$ defined by
$$\slo(i)=\frac{\dim_l(i)}{\dim_r(i)}.$$
The \emph{slope}  of an object $X$ of $\cc$ is the morphism $\SLO_X \co X \to X$ defined as
$$
\SLO_X=\sum_{\alpha \in A} \slo(i_\alpha)\, q_\alpha p_\alpha,
$$
where  $(p_\alpha\co X \to i_\alpha, q_\alpha \co i_\alpha \to X)_{\alpha \in A}$ is a decomposition of $X$ as a sum of scalar objects, that is, a family of  pairs of morphisms such that $i_\alpha$ is scalar for every $\alpha \in A$, $p_\alpha \,q_\beta = \delta_{\alpha,\beta}\,\id_{i_\alpha}$ for all $\alpha,\beta \in A$, and $\id_X=\sum_{\alpha \in A} q_\alpha p_\alpha$. The morphism $\SLO_X$ does not depend on the choice of the decomposition of $X$ into scalar objects. Note that $\SLO_X$ is invertible with inverse
$$
\SLO_X^{-1}=\sum_{\alpha \in A} \slo(i_\alpha)^{-1}\, q_\alpha p_\alpha.
$$

The family
$\SLO=\{\SLO_X\co X \to X\}_{X\in\cc}$ is a monoidal natural automorphism of the identity functor $1_\cc$ of $\cc$, called the \emph{slope operator} of~$\cc$. In particular
\begin{equation*}
\SLO_{Y}f=f\SLO_{X}\quad \text{and} \quad\SLO_{X\otimes Y}= \SLO_X\otimes \SLO_Y
\end{equation*}
for all objects $X,Y$ of $\cc$ and all morphism $f \co X \to Y$.
The slope operator relates the left and right traces:
for any endomorphism $f$ of an object of $\cc$,
\begin{equation}\label{eq-traces-slope}
\tr_l(f)=\tr_r(f \,\SLO_X).
\end{equation}
Note that $\cc$ is spherical if and only its slope operator is the identity.

\begin{lem}\label{firstlemmagraphic}  Let $I$ be a representative set of scalar objects of $\cc$. Then:
\begin{enumerate}
\labela
\item For any object $X$ of $\cc$,
\begin{equation*}
\psfrag{X}[Bl][Bl]{\scalebox{.9}{$X$}}
\psfrag{i}[Bl][Bl]{\scalebox{.9}{$i$}}
\psfrag{Y}[Bl][Bl]{\scalebox{.9}{$X$}}
\psfrag{X}[Bl][Bl]{\scalebox{.9}{$X$}}
\psfrag{j}[Bl][Bl]{\scalebox{.9}{$j$}}
\sum_{j \in I} \;\,\rsdraw{.45}{.9}{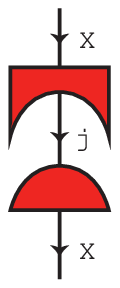} \; = \; \rsdraw{.45}{.9}{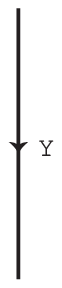}\;.
\end{equation*}

\item   For $i,j \in I$ and $X,Y$  objects of $\cc$,
\begin{align*}
&\psfrag{X}[Bl][Bl]{\scalebox{.9}{$Y$}}
\psfrag{k}[Br][Br]{\scalebox{.9}{$X$}}
\psfrag{i}[Br][Br]{\scalebox{.9}{$i$}}
\psfrag{j}[Bl][Bl]{\scalebox{.9}{$j$}}
\rsdraw{.45}{.9}{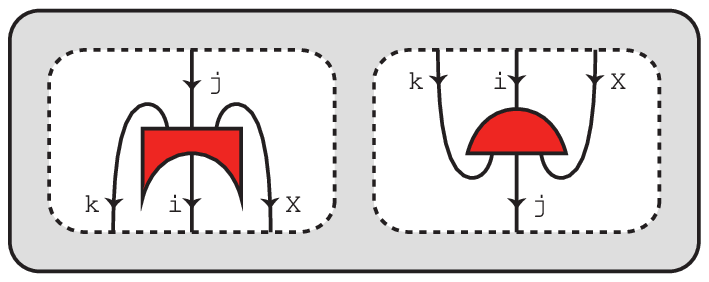}\\[.5em]
&\psfrag{X}[Bl][Bl]{\scalebox{.9}{$Y$}}
\psfrag{F}[Bc][Bc]{\scalebox{.8}{$\SLO_X^{-1}$}}
\psfrag{k}[Br][Br]{\scalebox{.9}{$X$}}
\psfrag{i}[Bl][Bl]{\scalebox{.9}{$i$}}
\psfrag{j}[Bl][Bl]{\scalebox{.9}{$j$}}
\psfrag{T}[Bc][Bc]{\scalebox{.9}{$\SLO_Y$}}
= \, \frac{\dim_r(i)}{\dim_r (j)}\;\,
\rsdraw{.45}{.9}{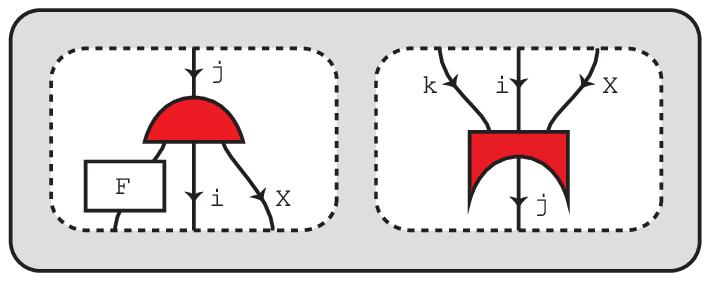}\,.
\end{align*}

 \item For $i \in I$ and $X,Y$ objects of $\cc$,
\begin{align*}
\psfrag{X}[Br][Br]{\scalebox{.9}{$X$}}
\psfrag{Y}[Bl][Bl]{\scalebox{.9}{$Y$}}
\psfrag{i}[Br][Br]{\scalebox{.9}{$j$}}
\psfrag{j}[Bl][Bl]{\scalebox{.9}{$i$}}
\rsdraw{.45}{.9}{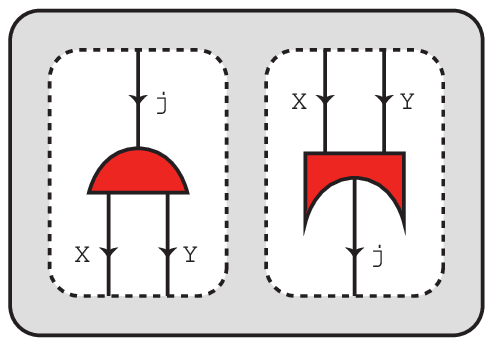} &= \;
\sum_{j \in I} \;\, \psfrag{X}[Br][Br]{\scalebox{.9}{$X$}}
\psfrag{Y}[Bl][Bl]{\scalebox{.9}{$Y$}}
\psfrag{i}[Br][Br]{\scalebox{.9}{$j$}}
\psfrag{j}[Bl][Bl]{\scalebox{.9}{$i$}} \rsdraw{.45}{.9}{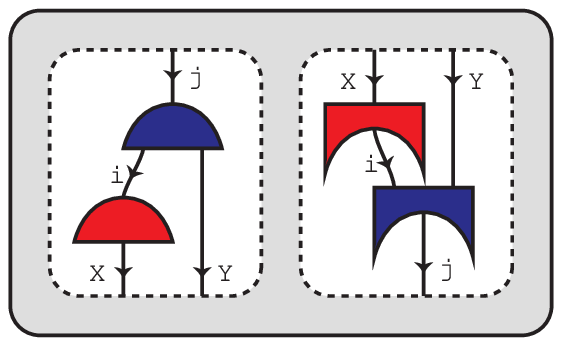} \\[.5em]
& = \; \sum_{j \in I} \;\, \psfrag{X}[Br][Br]{\scalebox{.9}{$X$}}
\psfrag{Y}[Bl][Bl]{\scalebox{.9}{$Y$}}
\psfrag{i}[Bl][Bl]{\scalebox{.9}{$j$}}
\psfrag{j}[Bl][Bl]{\scalebox{.9}{$i$}} \rsdraw{.45}{.9}{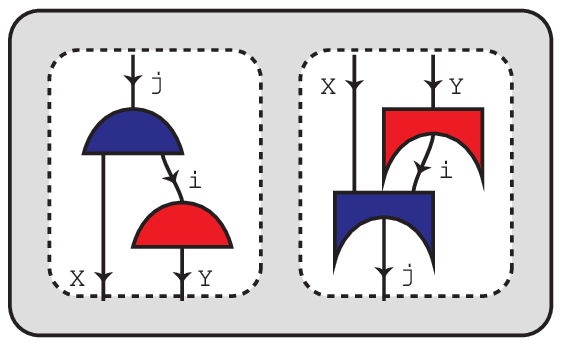}
\end{align*}
provided  there are no $j$-colored strands  in the gray area.

\item For all $i,j \in I$,
\begin{equation*}
\psfrag{i}[Br][Br]{\scalebox{.9}{$i$}}
\psfrag{j}[Bl][Bl]{\scalebox{.9}{$j$}}
 \rsdraw{.45}{.9}{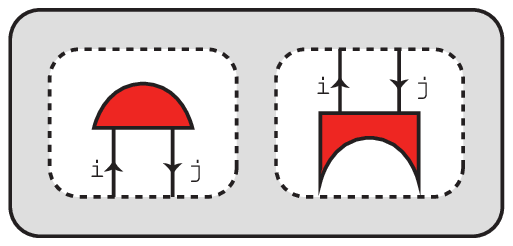}\, = \, \frac{\delta_{i,j}}{\dim_l (i)}\; \psfrag{i}[Bc][Bc]{\scalebox{.9}{$i$}}
\rsdraw{.45}{.9}{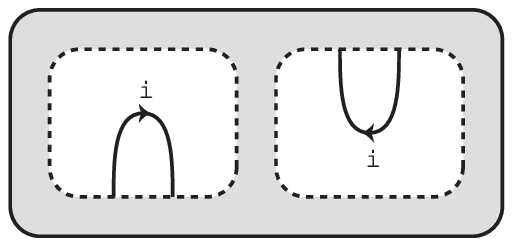} \;,
\end{equation*}
\begin{equation*}
\psfrag{i}[Br][Br]{\scalebox{.9}{$i$}}
\psfrag{j}[Bl][Bl]{\scalebox{.9}{$j$}}
 \rsdraw{.45}{.9}{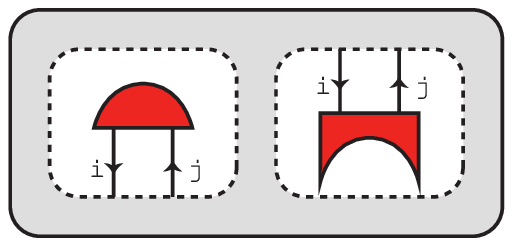}\, = \, \frac{\delta_{i,j}}{\dim_r (i)}\; \psfrag{i}[Bc][Bc]{\scalebox{.9}{$i$}}
\rsdraw{.45}{.9}{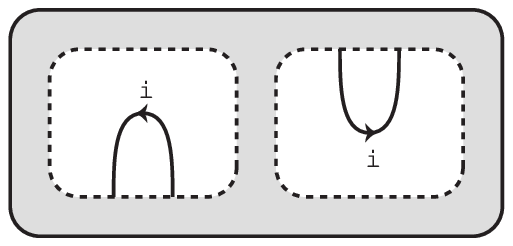}\;.
\end{equation*}

\end{enumerate}
\end{lem}

\begin{proof} Part
  (a) is a direct consequence of the definitions. Let us prove  Part (b).  Let
$(p_\alpha,q_\alpha)_{\alpha \in A}$ be an $i$-decomposition of $X^*
\otimes j \otimes Y^*$. For $\alpha,\beta \in A$,  set
\begin{equation*}
\psfrag{p}[Bc][Bc]{\scalebox{1}{$p_\alpha$}}
\psfrag{q}[Bc][Bc]{\scalebox{1}{$q_\alpha$}}
\psfrag{X}[Bl][Bl]{\scalebox{.9}{$Y$}}
\psfrag{k}[Br][Br]{\scalebox{.9}{$X$}}
\psfrag{i}[Br][Br]{\scalebox{.9}{$i$}}
\psfrag{j}[Bl][Bl]{\scalebox{.9}{$j$}}
\psfrag{F}[Bc][Bc]{\scalebox{.9}{$\SLO_X$}}
P_\alpha= \frac{\dim_r(j)}{\dim_r (i)}\;\,\rsdraw{.5}{.9}{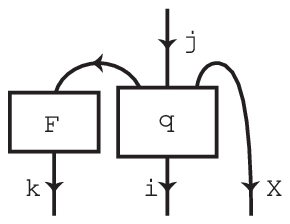},  \quad Q_\alpha=\rsdraw{.5}{.9}{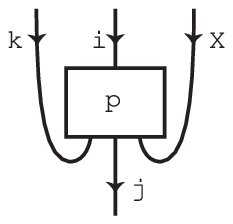},
\quad   f_{\alpha,\beta}=\; \, \psfrag{X}[Br][Br]{\scalebox{.9}{$Y$}}
\psfrag{p}[Bc][Bc]{\scalebox{1}{$p_\beta$}}
\psfrag{k}[Br][Br]{\scalebox{.9}{$X$}}
\psfrag{i}[Br][Br]{\scalebox{.9}{$i$}}
\psfrag{j}[Bl][Bl]{\scalebox{.9}{$j$}} \rsdraw{.5}{.9}{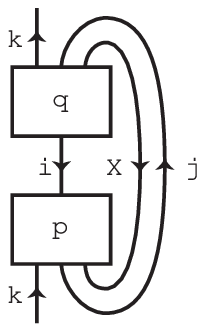}\;\,.
\end{equation*}
We need to prove that $(P_\alpha,Q_\alpha)_{\alpha \in A}$ is a $j$-decomposition of $X \otimes i \otimes Y$.
Let $\alpha,\beta \in A$. Since $(\SLO_X)^*=\SLO_{X^*}^{-1}$ and using \eqref{eq-traces-slope}, we obtain
\begin{align*}
P_\alpha Q_\beta & =\frac{\tr_r(P_\alpha Q_\beta)}{\dim_r(j)} \, \id_j = \frac{\tr_l(f_{\alpha,\beta} \SLO_{X^*}^{-1})}{\dim_r(i)}  \, \id_j = \frac{\tr_r(f_{\alpha,\beta})}{\dim_r(i)}  \, \id_j\\
& = \frac{\tr_r(q_\alpha p_\beta)}{\dim_r(i)}\, \id_j= \frac{\tr_r(p_\beta q_\alpha )}{\dim_r(i)} \, \id_j
 = \frac{\tr_r(\delta_{\alpha,\beta}\, \id_i)}{\dim_r(i)} \, \id_j=\delta_{\alpha,\beta}\, \id_j.
\end{align*}
We conclude using that ${\rm  {card}} (A)=\nu_i(X^* \otimes j \otimes Y^*)=\nu_j(X
\otimes i \otimes Y)$.

Part (c) reflects the canonical isomorphisms
\begin{align*}
\Hom_\cc(X \otimes
Y,i)&\cong\bigoplus_{j \in I} \Hom_\cc(X,j) \otimes_\kk \Hom_\cc(j \otimes Y,i)\\
&\cong\bigoplus_{j \in I}  \Hom_\cc(X \otimes j,i)  \otimes_\kk \Hom_\cc(Y,j),
\end{align*}
and Part (d) is a direct consequence of the duality axioms.
\end{proof}

\subsection{Proof of Theorem~\ref*{thm-center-modular} and of the integrality of $\Lambda$}\label{sect-proof-modular}
Recall that $\zz(\cc)$ is a braided pivotal category which has a coend $(C,\sigma)$ with $C=\bigoplus_{i,j \in I} i^* \otimes j^* \otimes i \otimes j$. Therefore its dimension is well-defined and
\begin{gather*}
\dim \zz(\cc)=\dim_l(C,\sigma)=\dim_l(C)=\dim_l\left (\sum_{i,j \in I} i^* \otimes j ^* \otimes i \otimes j \right ) \\ = \sum_{i,j \in I} \dim_l(i^*) \dim_l(j^*) \dim_l(i) \dim_l(j)\\= \left (\sum_{i \in I} \dim_r(i)  \dim_l(i) \right)\left (\sum_{j \in I}  \dim_r(j) \dim_l(j) \right)=\dim(\cc)^2.
\end{gather*}

Let us prove that the canonical pairing of the coend $(C,\sigma)$ is non-degenerate. Define the morphism $\lambda \co C \to \un$ as follows and recall  the definition of the morphism $\Lambda\co \un \to C$ of \eqref{eq-def-Lambda}:
$$
 \lambda =
 \psfrag{i}[Bl][Bl]{\scalebox{.85}{$i$}}
\sum_{i\in I} \,\dim_r(i) \;\;\; \rsdraw{.30}{.9}{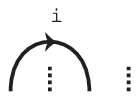}
\qquad \text{and} \qquad
\Lambda=\sum_{j \in I} \, \dim_r(j) \,  \psfrag{i}[Bl][Bl]{\scalebox{.85}{$j$}}
 \;\rsdraw{.5}{.9}{C-Lambda}\;.
$$
 Firstly, $\Lambda$ is a morphism in $\zz(\cc)$ from $\un_{\zz(\cc)}=(\un,\id)$ to $(C,\sigma)$. Indeed, using the description of the half braiding $\sigma$ given in \eqref{eq-half-braiding-coend-ZC}, we obtain that for any object~$Y$ of~$\cc$,
\begin{gather*}
\sigma_Y(\Lambda \otimes \id_Y)=
 \psfrag{i}[Bl][Bl]{\scalebox{.85}{$i$}}
 \psfrag{j}[Bl][Bl]{\scalebox{.85}{$j$}}
 \psfrag{k}[Bl][Bl]{\scalebox{.85}{$k$}}
 \psfrag{l}[Bl][Bl]{\scalebox{.85}{$\ell$}}
 \psfrag{n}[Bl][Bl]{\scalebox{.85}{$n$}}
 \psfrag{X}[Bl][Bl]{\scalebox{.85}{$Y$}}
 \!\!\sum_{j,k,\ell,n\in I} \dim_r(j) \;\;\; \rsdraw{.50}{.9}{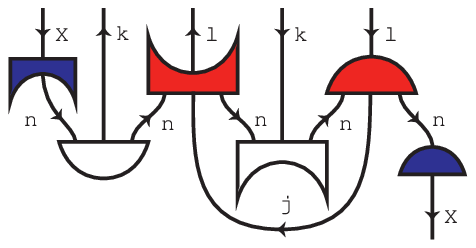} \\
 =
 \psfrag{i}[Bl][Bl]{\scalebox{.85}{$i$}}
 \psfrag{j}[Bl][Bl]{\scalebox{.85}{$j$}}
 \psfrag{k}[Bl][Bl]{\scalebox{.85}{$k$}}
 \psfrag{l}[Bl][Bl]{\scalebox{.85}{$\ell$}}
 \psfrag{n}[Bl][Bl]{\scalebox{.85}{$n$}}
 \psfrag{X}[Bl][Bl]{\scalebox{.85}{$Y$}}
 \!\!\sum_{j,k,\ell,n\in I}  \frac{\dim_r(\ell)}{\slo(n)} \;\;\; \rsdraw{.50}{.9}{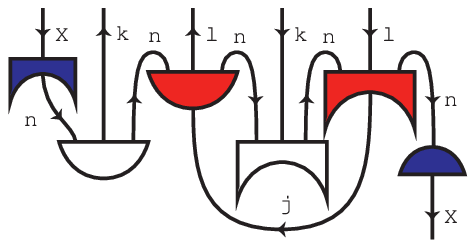} \quad \text{by Lemma~\ref{firstlemmagraphic}(b)}\\
 =
 \psfrag{i}[Bl][Bl]{\scalebox{.85}{$i$}}
 \psfrag{j}[Bl][Bl]{\scalebox{.85}{$j$}}
 \psfrag{k}[Bl][Bl]{\scalebox{.85}{$k$}}
 \psfrag{l}[Bl][Bl]{\scalebox{.85}{$\ell$}}
 \psfrag{n}[Bl][Bl]{\scalebox{.85}{$n$}}
 \psfrag{X}[Bl][Bl]{\scalebox{.85}{$Y$}}
 \!\!\sum_{k,\ell,n\in I}  \frac{\dim_r(\ell)}{\slo(n)}\;\;\; \rsdraw{.50}{.9}{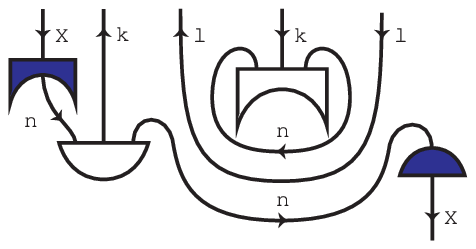} \quad \text{by Lemma~\ref{firstlemmagraphic}(a)}\\
 =
 \psfrag{i}[Bl][Bl]{\scalebox{.85}{$i$}}
 \psfrag{j}[Bl][Bl]{\scalebox{.85}{$j$}}
 \psfrag{k}[Bl][Bl]{\scalebox{.85}{$k$}}
 \psfrag{l}[Bl][Bl]{\scalebox{.85}{$\ell$}}
 \psfrag{n}[Bl][Bl]{\scalebox{.85}{$n$}}
 \psfrag{X}[Bl][Bl]{\scalebox{.85}{$Y$}}
 \sum_{\ell,n\in I}  \frac{\dim_r(\ell)}{\slo(n)} \;\;\; \rsdraw{.50}{.9}{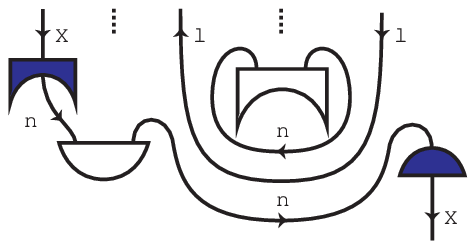} \\
  =
 \psfrag{i}[Bl][Bl]{\scalebox{.85}{$i$}}
 \psfrag{j}[Bl][Bl]{\scalebox{.85}{$j$}}
 \psfrag{k}[Bl][Bl]{\scalebox{.85}{$k$}}
 \psfrag{l}[Bl][Bl]{\scalebox{.85}{$\ell$}}
 \psfrag{n}[Bl][Bl]{\scalebox{.85}{$n$}}
 \psfrag{X}[Bl][Bl]{\scalebox{.85}{$Y$}}
 \sum_{\ell,n\in I}  \dim_r(\ell) \;\;\; \rsdraw{.50}{.9}{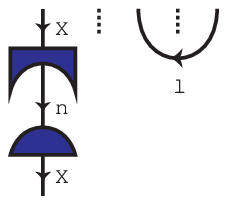} \quad \text{by Lemma~\ref{firstlemmagraphic}(d)}\\
  = \id_Y \otimes \Lambda \quad \text{by Lemma~\ref{firstlemmagraphic}(a).}
\end{gather*}
Secondly, $\lambda$ and $\Lambda$ satisfy  $\omega(\id_C \otimes \Lambda) = \lambda = \omega(\Lambda \otimes \id_C)$.
Indeed, using the description of the canonical pairing $\omega$ given in Figure~\ref{fig-struct-coend}, we obtain
\begin{gather*}
\omega(\id_C \otimes \Lambda)=
 \psfrag{i}[Bl][Bl]{\scalebox{.85}{$i$}}
 \psfrag{j}[Bl][Bl]{\scalebox{.85}{$j$}}
 \psfrag{k}[Bl][Bl]{\scalebox{.85}{$k$}}
 \psfrag{l}[Bl][Bl]{\scalebox{.85}{$\ell$}}
 \psfrag{n}[Bl][Bl]{\scalebox{.85}{$n$}}
 \psfrag{X}[Bl][Bl]{\scalebox{.85}{$Y$}}
 \!\!\sum_{i,j,\ell\in I} \dim_r(\ell)  \rsdraw{.50}{.9}{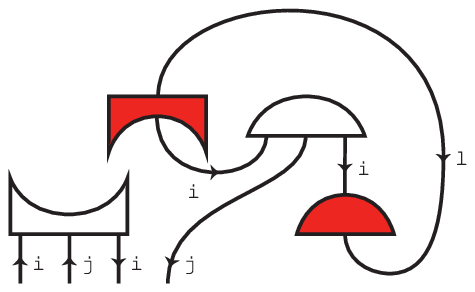}
  \end{gather*}
 \begin{gather*}
 =
 \psfrag{i}[Bl][Bl]{\scalebox{.85}{$i$}}
 \psfrag{j}[Bl][Bl]{\scalebox{.85}{$j$}}
 \psfrag{k}[Bl][Bl]{\scalebox{.85}{$k$}}
 \psfrag{l}[Bl][Bl]{\scalebox{.85}{$\ell$}}
 \psfrag{n}[Bl][Bl]{\scalebox{.85}{$n$}}
 \psfrag{X}[Bl][Bl]{\scalebox{.85}{$Y$}}
 \!\!\sum_{i,j,\ell\in I}   \dim_r(\ell) \rsdraw{.50}{.9}{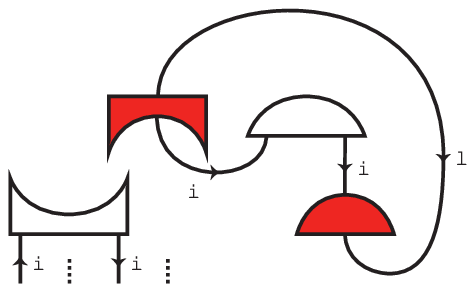} \\
 =
 \psfrag{i}[Bl][Bl]{\scalebox{.85}{$i$}}
 \psfrag{j}[Bl][Bl]{\scalebox{.85}{$j$}}
 \psfrag{k}[Bl][Bl]{\scalebox{.85}{$k$}}
 \psfrag{l}[Bl][Bl]{\scalebox{.85}{$\ell$}}
 \psfrag{n}[Bl][Bl]{\scalebox{.85}{$n$}}
 \psfrag{X}[Bl][Bl]{\scalebox{.85}{$Y$}}
 \sum_{i,\ell\in I}  \frac{\dim_r(\ell)}{\dim_l(i)} \;\;\; \rsdraw{.50}{.9}{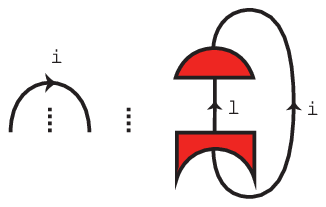} \quad \text{by Lemma~\ref{firstlemmagraphic}(d)}\\
   =
 \psfrag{i}[Bl][Bl]{\scalebox{.85}{$i$}}
 \psfrag{j}[Bl][Bl]{\scalebox{.85}{$j$}}
 \psfrag{k}[Bl][Bl]{\scalebox{.85}{$k$}}
 \psfrag{l}[Bl][Bl]{\scalebox{.85}{$\ell$}}
 \psfrag{n}[Bl][Bl]{\scalebox{.85}{$n$}}
 \psfrag{X}[Bl][Bl]{\scalebox{.85}{$Y$}}
\sum_{i,\ell\in I}  \frac{\dim_r(\ell)}{\dim_l(i)} \;\delta_{\ell,i^*}\;\; \rsdraw{.30}{.9}{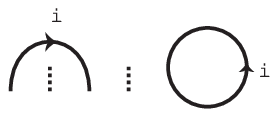} \;\,
    = \,
 \psfrag{i}[Bl][Bl]{\scalebox{.85}{$i$}}
 \psfrag{j}[Bl][Bl]{\scalebox{.85}{$j$}}
 \psfrag{k}[Bl][Bl]{\scalebox{.85}{$k$}}
 \psfrag{l}[Bl][Bl]{\scalebox{.85}{$\ell$}}
 \psfrag{n}[Bl][Bl]{\scalebox{.85}{$n$}}
 \psfrag{X}[Bl][Bl]{\scalebox{.85}{$Y$}}
\sum_{i\in I} \; \dim_r(i) \;\;\; \rsdraw{.30}{.9}{C-integ} \;=\lambda,
\end{gather*}
and similarly
\begin{gather*}
\omega(\Lambda \otimes \id_C)=
 \psfrag{i}[Bl][Bl]{\scalebox{.85}{$i$}}
 \psfrag{j}[Bl][Bl]{\scalebox{.85}{$j$}}
 \psfrag{k}[Bl][Bl]{\scalebox{.85}{$k$}}
 \psfrag{l}[Bl][Bl]{\scalebox{.85}{$\ell$}}
 \psfrag{n}[Bl][Bl]{\scalebox{.85}{$n$}}
 \psfrag{X}[Bl][Bl]{\scalebox{.85}{$Y$}}
 \!\!\sum_{j,k,\ell\in I} \dim_r(j)  \; \rsdraw{.50}{.9}{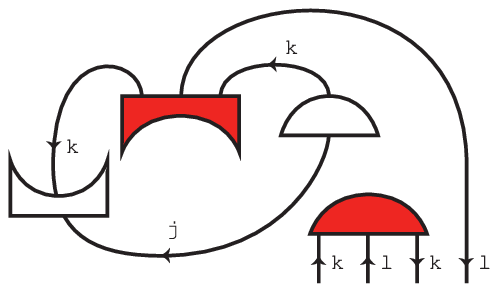} \\
 =
 \psfrag{i}[Bl][Bl]{\scalebox{.85}{$i$}}
 \psfrag{j}[Bl][Bl]{\scalebox{.85}{$j$}}
 \psfrag{k}[Bl][Bl]{\scalebox{.85}{$k$}}
 \psfrag{l}[Bl][Bl]{\scalebox{.85}{$\ell$}}
 \psfrag{n}[Bl][Bl]{\scalebox{.85}{$n$}}
 \psfrag{X}[Bl][Bl]{\scalebox{.85}{$Y$}}
 \!\sum_{j,k\in I}   \dim_r(j) \; \delta_{j,k^*} \;\; \rsdraw{.50}{.9}{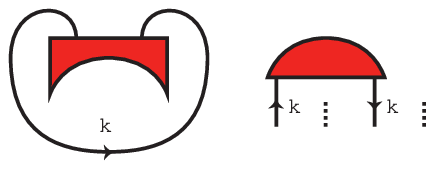} \\
   =
 \psfrag{i}[Bl][Bl]{\scalebox{.85}{$k$}}
 \psfrag{j}[Bl][Bl]{\scalebox{.85}{$j$}}
 \psfrag{k}[Bl][Bl]{\scalebox{.85}{$k$}}
 \psfrag{l}[Bl][Bl]{\scalebox{.85}{$\ell$}}
 \psfrag{n}[Bl][Bl]{\scalebox{.85}{$n$}}
 \psfrag{X}[Bl][Bl]{\scalebox{.85}{$Y$}}
\sum_{k\in I} \;\; \rsdraw{.30}{.9}{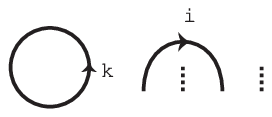} \quad \text{by Lemma~\ref{firstlemmagraphic}(d)}\\
    =
 \psfrag{i}[Bl][Bl]{\scalebox{.85}{$k$}}
\sum_{k\in I}  \; \dim_r(k) \;\;\; \rsdraw{.30}{.9}{C-integ} \;=\lambda.
\end{gather*}
This implies in particular that $\lambda$ is a morphism in $\zz(\cc)$ from $(C,\sigma)$ to $\un_{\zz(\cc)}$, since
 $\omega$ and $\Lambda$ are  morphisms in $\zz(\cc)$. 
Thirdly, $\lambda$ is a left integral for the Hopf algebra $(C,\sigma)$ in $\zz(\cc)$.  Indeed, using the description of the coproduct $\Delta$ and the unit $u$ given in Figure~\ref{fig-struct-coend}, we obtain
\begin{align*}
(\id_C \otimes \lambda)\Delta &=
 \psfrag{i}[Bl][Bl]{\scalebox{.85}{$i$}}
 \psfrag{j}[Bl][Bl]{\scalebox{.85}{$j$}}
 \psfrag{k}[Bl][Bl]{\scalebox{.85}{$k$}}
 \psfrag{l}[Bl][Bl]{\scalebox{.85}{$\ell$}}
 \psfrag{n}[Bl][Bl]{\scalebox{.85}{$n$}}
 \psfrag{X}[Bl][Bl]{\scalebox{.85}{$Y$}}
 \!\!\sum_{i,k,\ell,n\in I} \dim_r(k)  \; \rsdraw{.50}{.9}{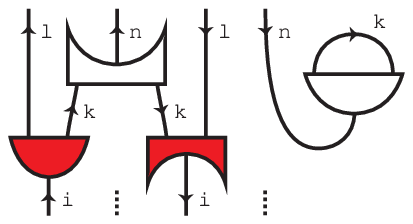} \\
 &=
 \psfrag{i}[Bl][Bl]{\scalebox{.85}{$i$}}
 \psfrag{j}[Bl][Bl]{\scalebox{.85}{$j$}}
 \psfrag{k}[Bl][Bl]{\scalebox{.85}{$k$}}
 \psfrag{l}[Bl][Bl]{\scalebox{.85}{$\ell$}}
 \psfrag{n}[Bl][Bl]{\scalebox{.85}{$n$}}
 \psfrag{X}[Bl][Bl]{\scalebox{.85}{$Y$}}
 \!\sum_{i,k,\ell\in I} \dim_r(k)  \; \rsdraw{.50}{.9}{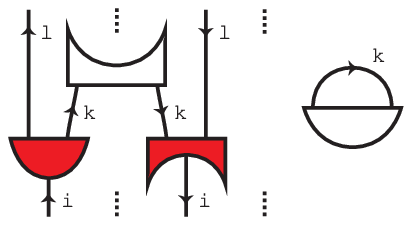}
  \end{align*}
\begin{align*}
  &=
 \psfrag{i}[Bl][Bl]{\scalebox{.85}{$i$}}
 \psfrag{j}[Bl][Bl]{\scalebox{.85}{$j$}}
 \psfrag{k}[Bl][Bl]{\scalebox{.85}{$k$}}
 \psfrag{l}[Bl][Bl]{\scalebox{.85}{$\ell$}}
 \psfrag{n}[Bl][Bl]{\scalebox{.85}{$n$}}
 \psfrag{X}[Bl][Bl]{\scalebox{.85}{$Y$}}
 \!\sum_{i,k,\ell\in I} \dim_r(k)  \; \rsdraw{.50}{.9}{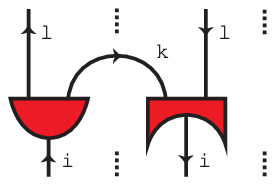} \quad \text{by Lemma~\ref{firstlemmagraphic}(d)}\\
    &=
 \psfrag{i}[Bl][Bl]{\scalebox{.85}{$i$}}
 \psfrag{j}[Bl][Bl]{\scalebox{.85}{$j$}}
 \psfrag{k}[Bl][Bl]{\scalebox{.85}{$k$}}
 \psfrag{l}[Bl][Bl]{\scalebox{.85}{$\ell$}}
 \psfrag{n}[Bl][Bl]{\scalebox{.85}{$n$}}
 \psfrag{X}[Bl][Bl]{\scalebox{.85}{$Y$}}
 \!\sum_{i,k,\ell\in I} \dim_r(i)  \; \rsdraw{.50}{.9}{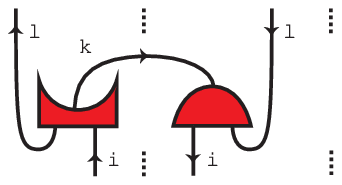} \quad \text{by Lemma~\ref{firstlemmagraphic}(b)}\\
    &=
 \psfrag{i}[Bl][Bl]{\scalebox{.85}{$i$}}
 \psfrag{j}[Bl][Bl]{\scalebox{.85}{$j$}}
 \psfrag{k}[Bl][Bl]{\scalebox{.85}{$k$}}
 \psfrag{l}[Bl][Bl]{\scalebox{.85}{$\ell$}}
 \psfrag{n}[Bl][Bl]{\scalebox{.85}{$n$}}
 \psfrag{X}[Bl][Bl]{\scalebox{.85}{$Y$}}
 \sum_{i,\ell\in I} \dim_r(i)  \; \rsdraw{.50}{.9}{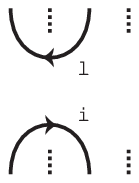} \quad \text{by Lemma~\ref{firstlemmagraphic}(a)}\\
  &=u\, \lambda.
\end{align*}
Since $\omega(\Lambda \otimes \Lambda)=\lambda \Lambda=\dim_r(\un)=1\in\kk$ is invertible, we conclude by Lemma~\ref{lem-non-degen} that $\omega$ is non-degenerate. Hence $\zz(\cc)$ is modular.

 Finally, let us prove that $\Lambda$ is a two-sided integral of $(C,\sigma)$ which is invariant under the antipode. The last part of Lemma~\ref{lem-non-degen} gives that $\Lambda$ is a right integral of $(C,\sigma)$. Using the description of the antipode $S$ of $(C,\sigma)$ given in Figure~\ref{fig-struct-coend}, we obtain
\begin{align*}
S\Lambda &=
 \psfrag{j}[Bl][Bl]{\scalebox{.85}{$j$}}
 \psfrag{k}[Bl][Bl]{\scalebox{.85}{$k$}}
 \psfrag{l}[Bl][Bl]{\scalebox{.85}{$\ell$}}
 \sum_{j,k,\ell\in I} \, \dim_r(j)  \; \rsdraw{.50}{.9}{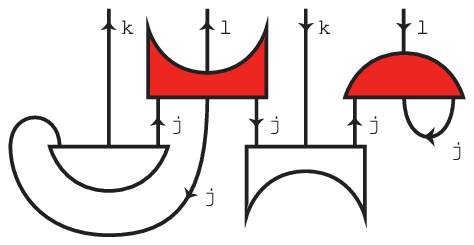} \\[.4em]
 &=
 \psfrag{j}[Bl][Bl]{\scalebox{.85}{$j$}}
 \psfrag{k}[Bl][Bl]{\scalebox{.85}{$k$}}
 \psfrag{l}[Bl][Bl]{\scalebox{.85}{$\ell$}}
 \sum_{j,k,\ell\in I} \, \dim_r(j)  \; \rsdraw{.50}{.9}{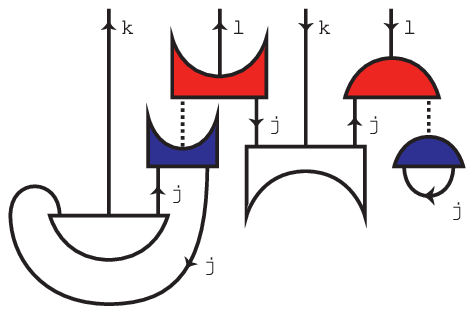} \quad \text{by Lemma~\ref{firstlemmagraphic}(c)}\\[.3em]
 &=
 \psfrag{k}[Br][Br]{\scalebox{.85}{$\ell$}}
 \psfrag{l}[Bl][Bl]{\scalebox{.85}{$\ell$}}
 \sum_{\phantom{j,}\ell\in I\phantom{k,}} \, \dim_r(\ell^*)  \, \rsdraw{.50}{.9}{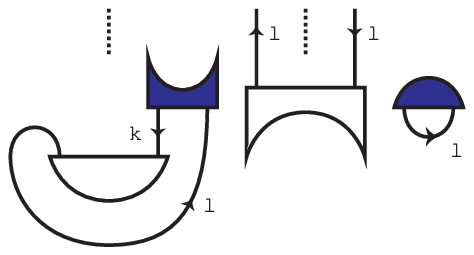} \\[.6em]
 &=
 \sum_{\phantom{j,}\ell\in I\phantom{k,}} \, \dim_r(\ell) \;
 \psfrag{i}[Bl][Bl]{\scalebox{.85}{$\ell$}}
 \;\rsdraw{.5}{.9}{C-Lambda} \quad \text{by Lemma~\ref{firstlemmagraphic}(d)}\\
 &= \; \Lambda.
\end{align*}
Hence $\Lambda$ is $S$-invariant. This implies in particular that $\Lambda$, being a right integral, is also a left integral. Hence $\Lambda$ is a $S$-invariant (two-sided) integral.

\subsection{Proof of Theorem~\ref*{thm-center-semisimple}}\label{sect-proof-semisimple}
Consider the Hopf monad $Z$ of Section~\ref{sect-ZC-monadic}. Recall from \cite{BV2} that the monad $Z$ is said to be \emph{semisimple} if any $Z$-module is a $Z$-linear retract of a free $Z$-module, that is, of $(Z(X), \mu_X)$ for some object $X$ of $\cc$. Since  the isomorphism $\Phi\co \cc^Z \to \zz(\cc)$ defined in \eqref{eq-iso-Phi} sends the free $Z$-module $(Z(X), \mu_X)$ to the free half braiding $\Phi(Z(X), \mu_X)=\ff(X)$, we need to prove that $\dim(\cc)$ is invertible if and only if $Z$ is semisimple. Now Theorem 6.5 of \cite{BV2} provides an analogue of Maschke's semisimplicity criterion for
Hopf monads: the Hopf monad $Z$ is semisimple if and only if
there exists a morphism $\alpha \co \un \to Z(\un)$ in $\cc$ such
that
\begin{equation}\label{eq-Maschke}
\mu_\un \alpha = \alpha Z_0 \quad \text{and} \quad Z_0\alpha = 1.
\end{equation}
Let $\alpha\co \un \to Z(\un)=\oplus_{i \in I} i^* \otimes i$ be a morphism  in $\cc$.
Since $\cc$ is a fusion category, $\alpha$ decomposes uniquely as
$
\alpha=\sum_{i\in I} \alpha_i \, \rcoev_i
$ where $\alpha_i \in \kk$.
From Figure~\ref{fig-strucZ-fusion}, we obtain
$$
\alpha Z_0= \sum_{j,k \in I}   \alpha_k\,
\psfrag{k}[Br][Br]{\scalebox{.9}{$k$}}
\psfrag{i}[Br][Br]{\scalebox{.9}{$i$}}
\psfrag{j}[Bl][Bl]{\scalebox{.9}{$j$}}
\rsdraw{.45}{.9}{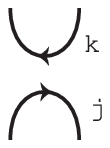}
\quad \text{and} \quad
\mu_\un Z(\alpha) = \sum_{i,j,k \in I}  \alpha_i\,\;
\psfrag{k}[Br][Br]{\scalebox{.9}{$k$}}
\psfrag{i}[Br][Br]{\scalebox{.9}{$i$}}
\psfrag{j}[Bl][Bl]{\scalebox{.9}{$j$}}
\rsdraw{.45}{.9}{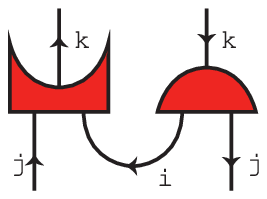}\;.
$$
Thus, by duality, $\alpha Z_0=\mu_\un Z(\alpha)$ if and only if
$$
\sum_{j,k \in I}   \alpha_k\;
\psfrag{k}[Bl][Bl]{\scalebox{.9}{$k$}}
\psfrag{j}[Bl][Bl]{\scalebox{.9}{$j$}}
\rsdraw{.45}{.9}{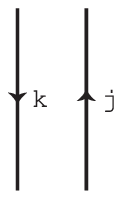}
\;\; = \sum_{i,j,k \in I}  \alpha_i\,\;
\psfrag{k}[Br][Br]{\scalebox{.9}{$k$}}
\psfrag{i}[Br][Br]{\scalebox{.9}{$i$}}
\psfrag{j}[Bl][Bl]{\scalebox{.9}{$j$}}
\rsdraw{.45}{.9}{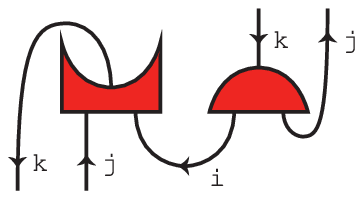}\; \text{in} \;\, \mathrm{End}_\cc\left ( \bigoplus_{k,j \in I} k \otimes j^* \right ).
$$
Now, for $j,k \in I$, by using Lemma~\ref{firstlemmagraphic}(b),
\begin{gather*}
\sum_{i\in I}  \alpha_i\,\;
\psfrag{k}[Br][Br]{\scalebox{.9}{$k$}}
\psfrag{i}[Br][Br]{\scalebox{.9}{$i$}}
\psfrag{j}[Bl][Bl]{\scalebox{.9}{$j$}}
\rsdraw{.45}{.9}{integ-lem3b} =
\sum_{i \in I}  \alpha_i\, \frac{\dim_r(k)}{\dim_r(i)}\,\;
\psfrag{k}[Br][Br]{\scalebox{.9}{$k$}}
\psfrag{i}[Bl][Bl]{\scalebox{.9}{$i$}}
\psfrag{j}[Bl][Bl]{\scalebox{.9}{$j$}}
\rsdraw{.45}{.9}{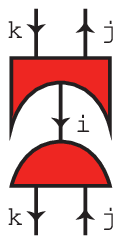} \;.
\end{gather*}
Therefore $\alpha Z_0=\mu_\un Z(\alpha)$ if and only if
\begin{equation}\label{eq-handleslide}
\alpha_k \, \id_{k \otimes j^*} =   \sum_{i \in I} \alpha_i\, \frac{\dim_r(k)}{\dim_r(i)} \,\;\psfrag{k}[Br][Br]{\scalebox{.9}{$k$}}
\psfrag{i}[Bl][Bl]{\scalebox{.9}{$i$}}
\psfrag{j}[Bl][Bl]{\scalebox{.9}{$j$}}
\rsdraw{.45}{.9}{integ-lem3c} \; \in \End_\cc(k \otimes j^*)
\;\text{
for all $k,j \in I$.}
\end{equation}
In particular, if $\alpha Z_0=\mu_\un Z(\alpha)$, then for any $i \in I$, setting $k = \un$ and $j = i^*$ we obtain $\alpha_i  = \alpha_\un \dim_r(i)$.
Conversely, if $\alpha_i = \alpha_\un \dim_r(i)$ for all $i \in I$, then \eqref{eq-handleslide} holds by Lemma~\ref{firstlemmagraphic}(a), and so $\alpha Z_0=\mu_\un Z(\alpha)$. In conclusion, $\alpha Z_0=\mu_\un Z(\alpha)$ if and only if $\alpha = \alpha_\un \kappa$, where
$$
\kappa = \sum_{i\in I} \dim_r(i) \, \rcoev_i \co \un \to Z(\un).
$$
In that case,
$$
Z_0\alpha=\alpha_\un Z_0\kappa=\sum_{i\in I}
\dim_r(i)Z_0\rcoev_i= \alpha_\un\sum_{i\in I}
\dim_r(i)\dim_l(i) = \alpha_\un\dim(\cc).
$$
Hence there exists $\alpha$ satisfying \eqref{eq-Maschke} if and only if $\dim(\cc)$ is invertible in $\kk$.
This concludes the proof of Theorem~\ref*{thm-center-semisimple}.

\subsection{Proof of Corollary~\ref*{cor-center-fusion}}\label{sect-proof-coro}
Let $\aaa$ be an abelian category. If $\aaa$ is semisimple (see Section~\ref{main-results}), then every object of $\aaa$ is projective\footnote{An object $P$  of $\aaa$ is \emph{projective} if the functor $\Hom_\aaa(P, -) \co \aaa \to \mathrm{Ab}$ is exact, where $\mathrm{Ab}$ is the category of abelian groups.}. The converse is true if in addition we assume that all objects of $\aaa$ have finite length\footnote{An object $A$ of $\aaa$ has \emph{finite length} if there exists a finite sequence of subobjects $A=X_0 \supsetneq X_1 \supsetneq \cdots \supsetneq X_n=\mathbf{0}$ such that each quotient
$X_i /X_{i + 1}$ is simple.}.

Assume $\kk$ is a field  and let $\cc$ be a pivotal fusion category over $\kk$. Then $\cc$ is abelian semisimple  and its objects have finite length. The center $\zz(\cc)$ of $\cc$ is then an abelian category and the forgetful functor $\uu\co \zz(\cc) \to \cc$ is \kt linear, faithful, and exact. This implies that
all objects of $\zz(\cc)$ have finite length  and the Hom spaces in $\zz(\cc)$ are finite dimensional.
As a result, $\zz(\cc)$ is semisimple if and only if all of its objects are projective.

We identify  $\zz(\cc)$ with the category $\cc^Z$ of $Z$\ti modules via the isomorphism \eqref{eq-iso-Phi}.
Recall from the proof of Theorem~\ref{thm-center-semisimple}  (see the beginning of Section~\ref{sect-proof-semisimple}) that the monad $Z$ is semisimple if and only if $\dim(\cc)$ is invertible in $\kk$.
The following lemma relates the notions of semisimplicity for monads and for categories.

\begin{lem}\label{lem-ssmonad-sscat}
Let $\cc$ be an abelian category and $T$ be a right exact monad on $\cc$, so that $\cc^T$ is abelian and the forgetful functor $U_T\co\cc^T \to \cc$ is exact. Then:
\begin{enumerate}
\labela
\item If all the $T$-modules are projective, then $T$ is semisimple.
\item If $T$ is semisimple and all the objects of $\cc$ are projective, then all the $T$-modules are projective.
\item If the objects of $\cc$ have finite length, then the same holds in $\cc^T$. If in addition  $\cc$ has finitely many isomorphy classes of simple objects, then so does $\cc^T$.
\end{enumerate}
\end{lem}

\begin{proof}
Let us prove Assertion (a). Denote by $F_T\co \cc \to \cc^T$ the free module functor (see Section~\ref{sect-Hopf-monoads}). Let $(M,r)$ be a  $T$\ti module. The action $r$ defines an epimorphism $F_T(M) \to (M,r)$ in $\cc^T$.
In particular, if $(M,r)$ is projective, it is a retract of $F_T(M)$. Therefore if all the $T$-modules are projective,  the monad $T$ is semisimple.

Let us prove Assertion (b). Note that if $X$ is a projective object of $\cc$, then $F_T(X)$ is a projective $T$-module. Indeed,
$\Hom_{\cc^T}(F_T(X), ?\,) \simeq \Hom_{\cc}(X, U_T)$ by adjunction, and $\Hom_{\cc}(X, U_T)$  is an exact functor when $X$ is projective. In particular, if all objects are projective in $\cc$ then all free $T$\ti modules are projective. If in addition $T$ is semisimple, then any $T$\ti module, being a retract of a free $T$\ti module, is projective.

Finally, let us prove Assertion (c). The first part results from the fact that $U_T$ is faithful exact. Now if $S$ is a simple object of $\cc^T$ and $\Sigma$ is a simple subobject of $U_T(S)$, then by adjunction the inclusion $\Sigma \subset U_T(S)$ defines a non-zero morphism $F_T(\Sigma) \to S$, which is an epimorphism
because $S$ is simple. This proves the second part of Assertion (c), because under the assumptions made there are finitely many possibilities for $\Sigma$, and each $F_T(\Sigma)$ has finitely many simple quotients.
\end{proof}

Assertion (a) of Corollary~\ref{cor-center-fusion} results immediately from the first two assertions of Lemma~\ref{lem-ssmonad-sscat}.

Let us prove Assertion (b). A fusion category over a field is semisimple. Now assume $\kk$ is algebraically closed. By Assertion~(a), we need to show that if  $\zz(\cc)$ is semisimple, then it is a fusion category. Assume $\zz(\cc)$ is semisimple. Since $\cc$ is fusion, by the third assertion of Lemma~\ref{lem-ssmonad-sscat}, the category
$\zz(\cc)$ has finitely many classes of simple objects and its objets have finite length. So each object of $\zz(\cc)$ is a finite direct sum of simple objects. Since the unit object of $\zz(\cc)$ is scalar and any simple object $S$  of $\zz(\cc)$ is scalar (because $\End(S)$ is a finite extension of $\kk$), we obtain that  $\zz(\cc)$ is a fusion category. This concludes the proof of Corollary~\ref{cor-center-fusion}.


\begin{thebibliography}{{Mac}98}

\bibitem[Ba]{Ba} Balsam, B., \emph{Turaev-Viro invariants as an extended TQFT II}, arXiv:1010.1222.

\bibitem[BW]{BW} Barrett, J., Westbury, B., \emph{Invariants of piecewise-linear
3-manifolds},  Trans. Amer. Math. Soc.  348  (1996),   3997--4022.

\bibitem[BV1]{BV2}
Brugui{\`e}res, A.,  Virelizier, A.,  \emph{Hopf monads}, Adv. in Math.  215  (2007), 679--733.

\bibitem[BV2]{BV3}
\bysame,  \emph{Quantum double of {H}opf monads and categorical centers},
Trans. Amer. Math. Soc.  364, Number 3 (2012), 1225--1279.

\bibitem[BV3]{BV4}
\bysame,  \emph{Categorical Centers and Reshetikhin-Turaev Invariants},  Acta Mathematica Vietnamica Volume 33, Number 3, 2008, 255--277.

\bibitem[BLV]{BLV}   Brugui{\`e}res, A., Lack, S., Virelizier, A., \emph{Hopf monads on monoidal categories},
Adv. in Math. 227 (2011), 745--800.

\bibitem[ENO]{ENO} Etingof, P.,  Nikshych, D.,  Ostrik, V.,  \emph{On fusion categories}, Ann.
of Math. (2)  162  (2005),   581--642.

\bibitem[JS]{JS}   Joyal, A.,  Street, R. \emph{Tortile Yang-Baxter operators in tensor categories},  J. Pure Appl. Algebra 71 (1991),   43--51.

\bibitem[Lyu]{Lyu2}
Lyubashenko, V., \emph{Invariants of $3$-manifolds and projective
  representations of mapping class groups via quantum groups at roots of
  unity}, Comm. Math. Phys. \textbf{172} (1995), no.~3, 467--516.

\bibitem[{Mac}]{ML1}
 {MacLane}, S., \emph{Categories for the working mathematician},
second ed.,
  Springer-Verlag, New York, 1998.

\bibitem[Ma1]{Maj}
 Majid, S., \emph{Representations, duals and quantum doubles of monoidal categories}, Rend. Circ. Mat. Palermo
Suppl. 26 (1991), 197--206.

\bibitem[Ma2]{Maj2}
Majid, S., \emph{Foundations of quantum group theory}, {Cambridge: Cambridge
  Univ. Press. xix, 607 p.}, 1995.

\bibitem[M\"{u}]{Mu} M\"{u}ger, M.,  \emph{From subfactors to categories and topology. II. The
quantum double of tensor categories and subfactors},   J. Pure Appl.
Algebra 180 (2003),   159--219.

\bibitem[Tu]{Tu1}
Turaev, V.,  \emph{Quantum invariants of knots and $3$-manifolds}, Walter de
  Gruyter \& Co., Berlin, 1994.

\bibitem[TVi]{TVi} Turaev, V., Virelizier, A., \emph{On two approaches to 3-dimensional TQFTs},  arXiv:1006.3501.


\end{thebibliography}
\end{document}